\documentclass[12pt,a4paper]{amsart}

\usepackage[centertags]{amsmath}
\usepackage{amsfonts,ifthen,latexsym,amssymb,amsthm, amscd}
\usepackage[utf8]{inputenc}
\usepackage[shortlabels]{enumitem}

\usepackage[usenames,dvipsnames]{color}

\usepackage[american]{babel} 

\usepackage{graphicx}

\usepackage{url}
\usepackage{hyperref}
\hypersetup{colorlinks=true,citecolor=blue,filecolor=blue,linkcolor=blue,urlcolor=blue}

\newcommand{\beq}{\begin{equation}}
\newcommand{\eeq}{\end{equation}}
\newtheorem{theorem}{Theorem}[section]
\newtheorem*{theorem*}{Theorem}
\newtheorem{lemma}[theorem]{Lemma}
\newtheorem*{lemma*}{Lemma}
\newtheorem{proposition}[theorem]{Proposition}

\newtheorem{cory}[theorem]{Corollary}
\newtheorem{definition}[theorem]{Definition}
\theoremstyle{definition}
\newtheorem{remark}[theorem]{Remark}
\newtheorem{remarks}[theorem]{Remarks}
\newtheorem*{remarks*}{Remarks}

\makeatletter
\newtheorem*{rep@theorem}{\rep@title}
\newcommand{\newreptheorem}[2]{%
\newenvironment{rep#1}[1]{%
 \def\rep@title{#2 \ref{##1}}%
 \begin{rep@theorem}}%
 {\end{rep@theorem}}}
\makeatother

\newreptheorem{theorem}{Theorem}
\newreptheorem{lemma}{Lemma}

\newcommand{\al}{{\alpha}}

\newcommand{\be}{{\beta}}

\newcommand{\eps}{{\varepsilon}}

\newcommand{\De}{{\Delta}}
\newcommand{\ga}{{\gamma}}
\newcommand{\Ga}{{\Gamma}}

\newcommand{\la}{{\lambda}}
\newcommand{\La}{{\Lambda}}
\newcommand{\si}{{\sigma}}
\newcommand{\Si}{{\Sigma}}
\renewcommand{\phi}{{\varphi}}

\newcommand\aut{\operatorname{Aut}}

\newcommand{\mat}{\operatorname{Mat}}

\newcommand\R{\mathbb R}
\newcommand\Q{\mathbb Q}
\newcommand\Z{\mathbb Z}
\newcommand\C{\mathbb C}
\newcommand\N{\mathbb N}

\newcommand\ZZ{\mathbf C}

\newcommand{\into}{\hookrightarrow}
\newcommand{\actson}{\curvearrowright}

\newcommand{\isimplied}{\Leftarrow}

\newcommand{\wt}[1]{{\widetilde {#1}}}
\newcommand{\wh}[1]{{\widehat {#1}}}
\newcommand{\cal}[1]{{\mathcal #1}}

\makeatletter
\newcommand*\if@single[3]{%
  \setbox0\hbox{${\mathaccent"0362{#1}}^H$}%
  \setbox2\hbox{${\mathaccent"0362{\kern0pt#1}}^H$}%
  \ifdim\ht0=\ht2 #3\else #2\fi
  }
\newcommand*\rel@kern[1]{\kern#1\dimexpr\macc@kerna}
\newcommand*\widebar[1]{\@ifnextchar^{{\wide@bar{#1}{0}}}{\wide@bar{#1}{1}}}
\newcommand*\wide@bar[2]{\if@single{#1}{\wide@bar@{#1}{#2}{1}}{\wide@bar@{#1}{#2}{2}}}
\newcommand*\wide@bar@[3]{%
  \begingroup
  \def\mathaccent##1##2{%
    \if#32 \let\macc@nucleus\first@char \fi
    \setbox\z@\hbox{$\macc@style{\macc@nucleus}_{}$}%
    \setbox\tw@\hbox{$\macc@style{\macc@nucleus}{}_{}$}%
    \dimen@\wd\tw@
    \advance\dimen@-\wd\z@
    \divide\dimen@ 3
    \@tempdima\wd\tw@
    \advance\@tempdima-\scriptspace
    \divide\@tempdima 10
    \advance\dimen@-\@tempdima
    \ifdim\dimen@>\z@ \dimen@0pt\fi
    \rel@kern{0.6}\kern-\dimen@
    \if#31
      \overline{\rel@kern{-0.6}\kern\dimen@\macc@nucleus\rel@kern{0.4}\kern\dimen@}%
      \advance\dimen@0.4\dimexpr\macc@kerna
      \let\final@kern#2%
      \ifdim\dimen@<\z@ \let\final@kern1\fi
      \if\final@kern1 \kern-\dimen@\fi
    \else
      \overline{\rel@kern{-0.6}\kern\dimen@#1}%
    \fi
  }%
  \macc@depth\@ne
  \let\math@bgroup\@empty \let\math@egroup\macc@set@skewchar
  \mathsurround\z@ \frozen@everymath{\mathgroup\macc@group\relax}%
  \macc@set@skewchar\relax
  \let\mathaccentV\macc@nested@a
  \if#31
    \macc@nested@a\relax111{#1}%
  \else
    \def\gobble@till@marker##1\endmarker{}%
    \futurelet\first@char\gobble@till@marker#1\endmarker
    \ifcat\noexpand\first@char A\else
      \def\first@char{}%
    \fi
    \macc@nested@a\relax111{\first@char}%
  \fi
  \endgroup
}
\makeatother
\newcommand{\wb}[1]{{\widebar {#1}}}

\newcommand*\mcapinn[2]{\vcenter{\hbox{$\mathsurround=0pt
  \ifx\displaystyle#1\textstyle\else#1\fi\bigcap$}}}

\newcommand*\mcupinn[2]{\vcenter{\hbox{$\mathsurround=0pt
  \ifx\displaystyle#1\textstyle\else#1\fi\bigcup$}}}

\newcommand\sscdot{{\cdot}}
\newcommand\ssminus{{\,-\,}}
\newcommand\sstimes{{\times}}
\newcommand\ssin{{\,\in\,}}

\newcommand\sscap{{\,\cap\,}}

\newcommand\ssactson{\,{\actson}}

\newcommand\chr{\operatorname{char}}

\newcommand\supp{{\operatorname{supp}}}

\newcommand\diag{\operatorname{Diag}}
\newcommand\vN{\operatorname{vN}}
\newcommand\dimvn{\dim_{\vN}}

\newcommand\lin{\operatorname{span}}
\newcommand\im{\operatorname{im}}

\renewcommand{\k}{\mathbf k}

\newcommand{\fun}{\operatorname{Fun}}
\def\mathclap#1{\text{\hbox to 0pt{\hss$\mathsurround=0pt#1$\hss}}}
\renewcommand{\ge}{\geqslant}
\renewcommand{\le}{\leqslant}
\renewcommand{\setminus}{\smallsetminus}

\begin{document}

\title{On computing homology gradients over finite fields}
\author{Łukasz Grabowski, Thomas Schick}

\address{Łukasz Grabowski - Department of Mathematics and Statistics,  Lancaster University, Lancaster LA1 4YF, United Kingdom}
\email{graboluk@gmail.com}
\address{Thomas Schick - Georg-August-Universit{\"a}t G{\"o}ttingen, Bunsenstr. 3, 37073 G{\"o}ttingen, Germany}
\email{thomas.schick@math.uni-goettingen.de}
\begin{abstract}
Recently the so-called Atiyah conjecture about $l^2$-Betti numbers has been disproved. The counterexamples were found using a specific method of computing the spectral measure of a matrix over a complex group ring. We show that in many situations the same method allows to compute homology gradients, i.e.~generalizations of $l^2$-Betti numbers to fields of arbitrary characteristic. As an application we point out that (i) the homology gradient over any field of characteristic different than $2$ can be an irrational number, and  (ii) there exists a finite CW-complex with the property that the homology gradients of its universal cover taken over different fields have infinitely many different values.
\end{abstract}
\maketitle
\newcommand{\Fa}{\Ga}

\setcounter{tocdepth}{1}
\tableofcontents

\section{Introduction}
Since their introduction by Atiyah \cite{Atiyah1976}, $l^2$-Betti numbers have been studied and used in many different contexts. We refer the reader to the very readable introductory article \cite{Eckmann_intro} and to the comprehensive book \cite{Lueck:Big_book} for more information.

For the current paper, the motivating question goes back to \cite{Atiyah1976}: are the $l^2$-Betti numbers rational?  This was popularized under the name \textit{Atiyah  conjecture} (e.g.~\cite[Chapter 10]{Lueck:Big_book}), although in \cite{Atiyah1976} it is stated as a problem. Only recently, motivated in part by the approach of \cite{Grigorchuk_Zuk2001} and \cite{Dicks_Schick}, Austin \cite{arxiv:austin-2009} showed that there exists a normal covering of a finite CW-complex with at least one irrational $l^2$-Betti number. Several improvements followed shortly afterwards (\cite{grabowski-on-turing-dynamical-systems-and-the-atiyah-problem} and \cite{arxiv:pichot_schick_zuk-2010}, \cite{arxiv:lehner_wagner-2010}, \cite{arXiv:grabowski-2010-2}).

Let us note in passing that the Atiyah conjecture remains open and very interesting in the case when the deck transformation group is torsion-free (see e.g.~\cite{Linnel:Atiyah_conjecture_for_free_groups}, 
\cite{dodziuk_linnell_mathai_schick_yates:approximating_l2_invariants}). 

The $l^2$-Betti numbers are invariants associated to a normal covering $\si\colon \wt S \to S$ of CW-complexes. Let us recall their definition in the case when $S$ is a finite  CW-complex  and the deck transformation group $\La$ of $\si$ is residually finite. Let $\La_1\supset \La_2\supset \,\ldots\,$ be a descending sequence of finite-index normal subgroups of $\La$ such that $\bigcap_{i=1}^\infty \La_i$ is the trivial subgroup,  and let  $S_1,S_2,\,\ldots\,$ be the corresponding sequence of normal covers of $S$. By a classical result of L\"uck \cite{lueck:aroximating_l2_invariants_by_their_finite}, the $i$-th $l^2$-Betti number of $\si$  can be defined as the limit
\beq\label{eq-lueck}
    \lim_{k\to \infty} \frac{\dim_\Q H_i(S_k;\Q)}{[\La:\La_k]}.
\eeq
This definition leads to an interesting generalization of $l^2$-Betti numbers: instead of the rational homology, one can try to use homology with coefficients in an arbitrary field $\k$. This was first suggested by Farber \cite{MR1625742}, and later studied extensively by Lackenby \cite{MR2496348} and others (e.g.~\cite{MR2966663},  \cite{MR2736326}).

If $\k$ is a field of positive characteristic, it is not known in general whether the limit in~\eqref{eq-lueck} exists when $\Q$ is replaced by $\k$. However, by results of Elek and Szab{\'o} \cite{MR2823074}, this is the case when $\La$ is an amenable group. Following Lackenby, we call the resulting limits the \textit{$\k$-homology gradients}. Throughout the article we will only consider the situation when $\La$ is amenable.

It is interesting to study the analogs of the Atiyah problem for $\k$-homology gradients. For example, a fundamental result of Linnell~\cite{Linnel:Atiyah_conjecture_for_free_groups} implies that if $\La$ is amenable and the group ring $\Q[\La]$ has no zero-divisors, then the $l^2$-Betti numbers are integers.  In~\cite{MR2736326}, the same result is obtained for $\k$-homology gradients, for an arbitrary field $\k$.

The motivation for the work presented here was to generalize the examples of
Austin and others   of ``exotic'' $l^2$-Betti numbers to the setting of
$\k$-homology gradients. At the heart of that work lies a method which translates the computation of  the spectral measure of a matrix over $\C[\La]$ (an a priori functional analytic task), to the computation of countably many spectra of finite matrices over $\C$. Section~\ref{sec-motivation} is devoted to presenting this computational tool and motivating why it is worthwhile to develop a similar tool for matrices over $\k[\La]$ for arbitrary fields $\k$. 

Our main technical result is such a computational tool for matrices over $\k[\La]$. It is contained in Theorem~\ref{thm:compu-tational-tool}. Most of Section~\ref{sec-definition} is  devoted to definitions and preliminary results needed in the proof of Theorem~\ref{thm:compu-tational-tool}.

Our remaining results are applications of Theorem~\ref{thm:compu-tational-tool}. They will be proved in Section~\ref{sec-app}.  

\begin{theorem}\label{thm:intro-finitely-generated}
If $\chr(\k) \neq 2$ then every non-negative real number is equal to the third $\k$-homology gradient of some normal covering of a finite CW-complex.
\end{theorem}

\begin{remarks}\label{rem-why}
\begin{enumerate}[(i),wide]
\item The deck transformation groups of the coverings which appear in Theorem~\ref{thm:intro-finitely-generated} are amenable but not, in general, residually finite. However, the $\k$-homology gradients can be defined whenever the deck transformation group is amenable (see \cite{cheeger_gromov:l2_cohomology_and_group_cohomology} and \cite{Elek_the_strong_approximation_conjecture_holds_for_amenable_groups}). We will state the general definition of $\k$-homology gradients in Section 3.

\item In \cite{grabowski-on-turing-dynamical-systems-and-the-atiyah-problem} and \cite{arxiv:pichot_schick_zuk-2010} it is shown that if $\k=\Q$, the expression "some normal covering" in Theorem~\ref{thm:intro-finitely-generated} can be replaced by "the universal covering". The reason why we cannot translate these results to an arbitrary field $\k$ is that their proofs employ the Higmann embedding theorem. As such, the examples in \cite{grabowski-on-turing-dynamical-systems-and-the-atiyah-problem} and \cite{arxiv:pichot_schick_zuk-2010} have deck transformation groups which are not amenable in general, and consequently their  $\k$-homology gradients are not known to be well-defined.
\end{enumerate}
\end{remarks}

\begin{theorem}\label{thm:intro-lamplighter3}
If $\chr(\k) \neq 2$ then there exists a finite CW-complex whose fundamental group
is residually finite $2$-step solvable, and such that the third $\k$-homology
gradient of its universal covering is equal to
\beq\label{eq-irr}
\frac1{64}- \frac18\sum_{k=1}^\infty \frac1{2^{k^2+4k+6}}.
\eeq
\end{theorem}

It is easy to see that the number~\eqref{eq-irr} is irrational.

The final result is an example of a finite CW-complex whose $\k$-homology gradients have infinitely many different values, depending on the characteristic of $\k$. For a field $\k$ of characteristic $\chr(\k)$ larger than $2$ let $o_{\k}(2)$ be the multiplicative order of $2$ in $\k^\ast$, let $L(\k)\subset \{2,3,\ldots,\chr(\k){+}1 \}$ be the finite set of minimal representatives in $\{2,3,\ldots,\chr(\k){+}1\}$ of the powers $2^0,2^1,2^2,\ldots$ taken modulo $\chr(\k)$, and for $x\ssin L(\k)$ let $r_{\k}(x)\ge 0$ be the smallest natural number such that $2^{r_{\k}(x)} =x \mod \chr(\k)$.

\begin{theorem}\label{thm-intro-positive} There exists a finite CW-complex $X$ whose
  fundamental group is residually finite $2$-step solvable, such that for all fields $\k$ with $\chr(\k)>2$ the third $\k$-homology
  gradient of the universal covering of $X$ is equal to
\beq\label{eq-rational-gradient}
1344\cdot \left(\frac{47}{64} + \frac{1}{128}\sscdot \frac{1}{2^{o_{\k}(2)}-1}
  + \frac{1}{64}\sscdot \frac{2^{\chr(\k)}}{2^{\chr(\k)}-1}\sscdot\frac{2^{o_{\k}(2)}}{2^{o_{\k}(2)}-1} \sscdot \sum_{x\in
    L(\k)} {2^{-(x+r_{\k}(x))}}\right),
\eeq

\end{theorem}

Since $L(\k)$ is a finite set, we see that \eqref{eq-rational-gradient} is a rational number. In \cite{arXiv:grabowski-2010-2} it is computed that the third $\Q$-homology gradient of $X$ is equal to the irrational number $1344\cdot (\frac{47}{64} +\frac{1}{64}\sum_{k=1}^\infty2^{-(k+2^{k})})$. It is not difficult to see that the numbers \eqref{eq-rational-gradient} converge to that number when  $\chr(\k)\to \infty$. In particular, the expression \eqref{eq-rational-gradient} represents infinitely many different numbers when $\k$ varies over different fields.

\begin{remarks}
\begin{enumerate}[(i),wide,partopsep=5pt, itemsep=5pt]

\item In analogy with the homology of a finite CW-complex, it has been conjectured by Andreas Thom that the $\k$-homology gradients stabilize as $\chr(\k)\to \infty$. Theorem~\ref{thm-intro-positive} disproves that conjecture.

\item Corresponding results for the case $\chr(\k)=2$ are shown in the G\"ottingen doctoral thesis \cite{NeumannDiss} of Johannes Neumann.

\item The results similar to the three theorems above hold also for fourth and higher $\k$-homology gradient. Furthermore, it is likely that with a bit more work one could find a finite CW-complex whose universal covering has an irrational second $\k$-homology gradient, for any field $\k$ such that $\chr(\k)\neq 2$. However, passing to the first $\k$-homology gradient is impossible with the methods used in this article. In particular, it is an open question whether there exist a field $\k$ and a finite CW-complex $X$ such that the first $\k$-homology gradient of the universal cover of $X$ is an irrational number. In other words, it is not known whether the $\k$-homology gradient of a finitely presented group can be irrational (see \cite{MR3100998} for results on $\k$-homology gradients of groups which are not finitely presented).
\end{enumerate}
\end{remarks}


\section{Motivation - characteristic $0$}\label{sec-motivation}

The computational tool which we will state in Theorem~\ref{thm:compu-tational-tool} works  over an arbitrary field $\k$ whose characteristic is not equal to $2$. To motivate it, we start by presenting a basic version in characteristic $0$. Various variants of it were used for spectral computations for example in \cite{MR1436310}, \cite{Dicks_Schick}, \cite{Lehner_Neuhauser_Woess:On_the_spectrum_of},  \cite{arxiv:austin-2009}, and \cite{arxiv:pichot_schick_zuk-2010}. A very general version is presented in \cite[Section 2]{grabowski-on-turing-dynamical-systems-and-the-atiyah-problem}.

Recall that when $\La$ is a countable group and $T$ is an element of the group ring $\C[\La]$, then $T$  induces a bounded operator on the Hilbert space $l^2(\La)$ of square-summable functions on $\La$. Let  $\zeta_e\in l^2(\La)$ be the indicator function of the neutral element $e\ssin \La$. The von Neumann dimension of the kernel of $T$ is defined as 
$$
\dimvn\ker T := \langle P\zeta_e,\zeta_e\rangle,
$$
where $P\colon l^2(\La)\to l^2(\La)$ is the orthogonal projection onto the kernel of $T$. 

Let $A$ be a countable abelian group, and let $X$ be its Pontryagin dual, i.e.~the set of group homomorphisms $A\to U(1)$, where $U(1)$ is the multiplicative group of the complex numbers on the unit circle. Let $\mu$ be the Haar measure on $X$, normalized so that $\mu(X)=1$. The Pontryagin duality induces an embedding 
\beq\label{eq-fourier-trans}
\wh{\ }\,\,\colon\C[A]\into L^\infty (X).
\eeq
We refer to \cite{Folland:A_course_in_abstract_harmonic_analysis} for more information on the Pontryagin duality.

Let $\Ga$ be a group and let $\Ga\actson A$ be an action by group automorphisms. The outcome of the
action of $\ga\ssin \Ga$ on $a\ssin A$ is denoted by $\ga.a$. The central dot
$\cdot$ is reserved for the group multiplication. The dual action
$\Ga\actson X$ is the unique continuous action which makes the embedding \eqref{eq-fourier-trans} equivariant. It is also denoted with the lower dot. To simplify the discussion, in this motivational section we assume that the dual action $\Ga\actson X$ is free on a subset of full measure. 

Let $\ga_1,\ldots, \ga_m$ be distinct elements of $\Ga$, and let $\chi_1,\ldots, \chi_m\ssin \C[A]$ be such that the functions $\wh\chi_1,\ldots, \wh\chi_m$ are indicator functions of some sets $X_1,\ldots, X_m\subset X$. Let $T \in \C[\Ga\ltimes A]$  be equal to the sum $\sum_{i=1}^m \ga_i \chi_i$.

Consider the oriented graph $\cal G$ defined as follows. The set of vertices of $\cal G$ is $X$, and there is an edge from $x$ to $y$ if for some $i$ we have $x\in X_i$ and $\ga_i.x=y$. 

Let $\cal G(x)$ be the connected component of $x$ in $\cal G$. Let $l^2(\cal G(x))$ be the Hilbert space spanned by the vertices of $\cal G(x)$. Let $M(x) \colon l^2(\cal G(x)) \to l^2(\cal G(x))$ be the adjacency operator on $\cal G(x)$, i.e.~the entry of the matrix of $M(x)$ corresponding to a pair of vertices $(v, w)$ is equal to $1$ if there is an edge from $v$ to $w$, and $0$ otherwise.

\begin{proposition}\label{prop-tool-free}
Let us assume that the set $\{x\ssin X\colon \cal G(x)\text{ is a finite graph}\}$ is of measure $1$. Then 
$$
\dimvn\ker (T) = \int_X \frac{\dim\ker M (x)}{|\cal G(x)|}\,d\mu(x).
$$\qed
\end{proposition}

The utility of this proposition comes from the fact that among the finite graphs $\cal G(x)$, $x\ssin X$, there are only countably many different ones, and they can be often computed explicitly. In such cases the above integral  decomposes as an explicit countable sum of kernel dimensions of \textit{finite-dimensional} matrices.

Let us give an example studied  in
\cite{Grigorchuk_Zuk2001} (by different methods) and in \cite{Dicks_Schick}. Let $\ZZ=\langle t\rangle$ be the infinite cyclic group and let $\ZZ_2=\langle u\rangle$ be the cyclic group of order $2$. The wreath product of $\ZZ_2$ and $\ZZ$ is denoted by $\ZZ_2\wr \ZZ$, thus in particular $\ZZ_2\wr \ZZ$ is a group of the form $\Ga\ltimes A$, where $\Ga:=\ZZ$ and $A := \oplus_\Z \ZZ_2$. Let $T\in \Z[\ZZ_2\wr \ZZ]$ be defined as $T:= t+t^{-1} + tu + ut^{-1}$. The whole procedure of applying Proposition~\ref{prop-tool-free} to compute $\dimvn\ker (T)$ is worked out in \cite[Section
3]{grabowski-on-turing-dynamical-systems-and-the-atiyah-problem}. The outcome
is as follows.

Let $G(k)$ be the oriented graph 
$$
\bullet\leftrightarrows \bullet \leftrightarrows \cdots \leftrightarrows \bullet \leftrightarrows \bullet \quad \text{(}k\text{ vertices)},
$$
and let $M(k)$ be the adjacency operator on $G(k)$.

The set of those points $x\ssin X$ for which $\cal G(x)$ is isomorphic to $G(k)$ has measure $\frac{k}{2^{k+1}}$. It is not difficult to check that $M(k)$ has $1$-dimensional kernel for odd $k$ and $0$-dimensional kernel otherwise. As such, Proposition \ref{prop-tool-free} allows us to compute $\dimvn\ker(T)$:
\beq\label{eq-gz-operator}
\dimvn\ker (T) = \sum_{k=1}^\infty \frac{k}{2^{k+1}} \frac{\dim\ker M(k)}{k} = \sum_{k=0}^\infty \frac{1}{2^{2k+1}} = \frac{1}{3}.
\eeq

Let $\k$ be a field  and let $\rho(T)\ssin \k[\Ga\ltimes A]$ be the group ring element which arises from $T$ by reducing the coefficients modulo $\chr(\k)$. In the next section we will explain how to define the number $\dim_{\k[\Ga\ltimes A]} \ker \rho(T)$, in such a way that $\dim_{\Q[\Ga\ltimes A]} \ker T = \dimvn\ker T$. As such, it is reasonable to expect that $\dim_{\k[\Ga\ltimes A]} \ker \rho(T)$ should be equal to
\beq\label{eq-mk}
\sum_{k=1}^\infty \frac{k}{2^{k+1}} \frac{\dim_\k\ker M(k;\k)}{k},
\eeq
where $M(k;\k)$ is the adjacency operator on $G(k)$ over the field $\k$. 

The main aim of the next two sections is to state and prove Theorem \ref{thm:compu-tational-tool} which is an ``arbitrary characteristic'' version of Proposition \ref{prop-tool-free}. A particular corollary of Theorem~\ref{thm:compu-tational-tool} is that, at least when $\chr(\k)\neq 2$, we have that $\dim_{\k[\Ga\ltimes A]} \ker \rho(T)$ is indeed equal to the expression~\eqref{eq-mk}.

\section{Definitions and preliminaries}\label{sec-definition}

The set $\{0,1,2,\ldots\}$ of natural numbers is denoted by $\N$. Given $n\ssin \N$, the set $\{1,2,\ldots, n\}$ is denoted by $[n]$, and the set $\{1,2,\ldots\}$ is denoted by $\N_+$. Given an action $\La\actson S$ of a group $\La$ on a set $S$, we denote the result of the action of $\la\ssin \La$ on $s\ssin S$ by $\la.s$. Similarly, for $\Phi\subset \La$ and $F\subset S$ we denote $\Phi.F :=\{\phi.s\colon \phi\ssin \Phi, s\ssin F\} \subset S$.

Given a set $S$ and a ring $R$ we let $R[S]$ be the abelian group of formal $R$-linear combinations of elements of $S$.  If $\La$ is a group then the \textit{group ring of $\La$ over $R$} is the ring whose additive group is $R[\La]$ and whose multiplication is the linear extension of the multiplication in $\La$. It will be also denoted by $R[\La]$. 

For $\Phi{\,\subset\,} \La$  we let $\pi_{\Phi}$ be the  projection $R[\La] \to R[\Phi]$ along the space $R[\La{\smallsetminus} \Phi]$. We use the same symbol $\pi_\Phi$ also for the projection $R[\La]^n \to R[\Phi]^n$. 

Given $T\ssin \mat({a{\times} b},R[\La])$ and $\Phi\subset \La$, let  
$$
	T_\Phi := \pi_\Phi\circ T \colon R[\Phi]^b \to R[\Phi]^a
$$
be the compression of $T$ to $R[\Phi]^b$. 

For $\Phi, \Si \subset \La$ we define $\Phi\sscdot \Si := \{\phi\sscdot \si\colon \phi \ssin \Phi, \si\ssin \Si\}$ and $\partial_\Si \Phi := \Si\sscdot \Phi \setminus \Phi$. 
A sequence $\Phi_1,\Phi_2,\ldots \subset \La$ of finite sets is called a \textit{F{\o}lner sequence} if for every finite subset $\Si\subset \La$ we have that
$$
    \frac{|\partial_\Si \Phi_i|}{|\Phi_i|} \xrightarrow{i\to\infty} 0.
$$
Recall that $\La$ is \textit{amenable} if there exists a F{\o}lner sequence in $\La$.

 \begin{definition}
   If $\La$ is an amenable group and $\k$ is a field then the \textit{kernel
     gradient} of $T\ssin \mat({a{\times} b},\k[\La]) $ is defined as
   \beq\label{eq-def-of-dimension} \dim_{\k[\La]} \ker T := \lim_{i\to\infty}
   \frac{\dim_\k \ker T_{\Phi_i}}{|\Phi_i|}.  \eeq where
   $\Phi_{1},\Phi_2,\,\ldots\,$ is a F{\o}lner sequence in $\La$.
 \end{definition}
The existence of this limit and its
independence of the choice of a  F{\o}lner sequence follows from the
Ornstein-Weiss lemma \cite{MR910005} (see \cite{MR2369443} for a detailed
discussion of that lemma). 

\begin{remark}
In this article the symbol $\dim_{\k[\La]} \ker$ should be understood as a single operator, not a composition of "$\dim_{\k[\La]}$" and "$\ker$". However, as shown in~\cite{NeumannDiss}, it is possible to develop a satisfactory dimension theory for very general modules over $\k[\Lambda]$. 
\end{remark}

\newcommand{\HG}{\operatorname{HG}}
Let $\si\colon \wt S \to S$ be a normal covering of CW-complexes such that $S$ is finite, and let $\La$ be the deck transformation group of $\si$. If $\La$ is amenable then for $i\ssin \N$ we define the \textit{$i$-th $\k$-homology gradient $\HG_i(\si; \k)$ } of $\si$ as follows. Let $\Phi_1,\Phi_2,\,\ldots\,$ be a F\o lner sequence in $\La$. Let $S_0\subset\wt S$ be a CW-subcomplex which is a fundamental domain for the action $\La\,{\actson}\, \wt S$. We let 
$$
\HG_i(\si; \k) := \lim_{k\to\infty} \frac{\dim_\k H_i(\Phi_k.S_0;\k)}{|\Phi_i|}.
$$
The independence from the choice of the fundamental domain $S_0$ and the F\o lner sequence $\Phi_1,\Phi_2,\ldots$ follows again from the Ornstein-Weiss lemma. Let us recall two well-known facts.

\begin{proposition}\label{prop-eck}
\begin{enumerate}[(i),wide]
\item Let $\si\colon \wt S \to S$ be a normal covering of CW-complexes such that $S$ is finite, and let $\La$ be the deck transformation group of $\si$. Let us suppose that $\La$ is amenable and residually finite and let $\La_1\supset \La_2\supset \,\ldots\,$ be a descending sequence of finite-index normal subgroups of $\La$ such that $\bigcap_{i=1}^\infty \La_i = \{1\}$. Finally, let $S_1:=\La_1\backslash \wt S,S_2:=\La_2\backslash \wt S,\ldots$ be the corresponding sequence of normal covers of $S$. Then we have
\beq
    \HG_i(\si; \k) = \lim_{k\to \infty} \frac{\dim_\k H_i(S_k;\k)}{[\La:\La_k]}.
\eeq
\item Let $\La$ be a countable amenable group, and let $T\ssin \Z[\La]$. There exists a connected finite CW-complex $S$ and a covering $\si\colon \wt S\to S$ whose deck transformation group is equal to $\La$, and such that for any field $\k$ we have 
$$
    \HG_3(\si; \k)  = \dim_{\k[\La]} \ker \rho(T),
$$
where $\rho(T) \ssin \k[\La]$ arises from $T$ by reducing the coefficient modulo $\chr(\k)$. Furthermore, if $\La$ is finitely presented then $\si$ can be taken to be the universal covering of $S$.
\end{enumerate}
\end{proposition}
\begin{proof}
The first item follows from~\cite{MR2823074}. The second item follows from~\cite[Section 3.10]{Eckmann_intro}.
\end{proof}

Let us also recall a standard lemma on F{\o}lner sequences in semi-direct products of amenable groups.

\begin{lemma}\label{lem-semi} Let $\Ga$ and $A$ be countable amenable groups, and let $\Ga \actson A$ be an action by group automorphisms. Let $\Phi_1,\Phi_2,\ldots$ be a F{\o}lner sequence in $\Ga$ and let $F_1,F_2,\ldots$ be a F{\o}lner sequence in $A$. Then there exists a subsequence $F_{i(1)},F_{i(2)},\ldots$ such that 
the sequence $\Phi_j\sscdot F_{i(j)}\subset \Ga\ltimes A$ is a F{\o}lner sequence in $\Ga\ltimes A$.
\end{lemma}

\begin{proof}
After passing to a subsequence of $F_i$ we may assume that for every finite set $S\subset A$ we have 
\beq\label{eq-z}
\frac{|\partial_{\Phi_j^{-1}.S}\,\, F_j|}{|F_j|}\,\, \xrightarrow{j\to \infty}\,\, 0.
\eeq
We will show that with this assumption the sequence  $\Phi_j\sscdot F_{j}$ is a F{\o}lner sequence.

First, it is easy to see  that it suffices to check that for all finite subsets $\Si$ of $\Ga$ and all finite subsets $\Si$ of $A$ the sequence
\beq\label{eq-f}
    \frac{|\partial_\Si \,\,\Phi_j\sscdot F_j| }{|\Phi_j \sscdot F_j|}
\eeq
converges to $0$. 

For $\Si\subset \Ga$, the convergence to $0$ of the sequence~\eqref{eq-f} follows directly from the fact that $\Phi_j$ is a F\o lner sequence in $\Ga$.

For $\Si\subset A$, let us fix $\eps\ssin(0,1)$, and let $j\in \N$ be such that $|(\Phi_j^{-1}.\Si)\sscdot F_j\setminus F_j|  \le \eps \sscdot |F_j|$. For $s\ssin \Si$, $\phi\in \Phi_j$ and $a\in F_j$ we have
$$
s\sscdot \phi\sscdot a = \phi\sscdot (\phi^{-1}.s)\sscdot a,
$$
and so the number of elements in $\Si\sscdot\Phi_j\sscdot F_j$ which are outside of $\Phi_j\sscdot F_j$ is at most $|\Phi_j|\sscdot \eps\sscdot |F_j|$. This easily implies the convergence of~\eqref{eq-f} to zero and finishes the proof.
\end{proof}

\subsection{Basic properties of kernel gradients}\label{subsec-basicprop}\mbox{}

The following basic properties of kernel gradients will be used in Section~\ref{sec-app}, in the proofs of Theorems~\ref{thm:intro-finitely-generated}, \ref{thm:intro-lamplighter3},  and~\ref{thm-intro-positive}. All of them are standard and we only give sketches of the proofs.

\begin{lemma}\label{lm-appendix-subgroups} 
Let $\Ga$ be an amenable group, let $\La \subset \Ga$ be a subgroup,  and let $T\ssin \mat(m\times n,\k[\La])$. Then $\dim_{\k[\La]}\ker T = \dim_{\k[\Ga]}\ker T$.
\end{lemma}

\begin{proof}[Sketch of Proof]
Let $\Phi_1,\Phi_2,\ldots$ be a F\o lner sequence in $\Ga$ and let $\ga_1,\ga_2,\ldots$ be right coset representatives of $\La$ in $\Ga$. We start by showing how to obtain a new F\o lner sequence $\Psi_1,\Psi_2,\ldots$ in $\Ga$ with the property that for any choice of indices $i(1),i(2),i(3),\ldots$ such that the sets $\Psi_1\cap \La\sscdot\ga_{i(1)}, \Psi_2 \cap \La\sscdot\ga_{i(2)},\ldots\,$ are all non-empty, we have that the sequence $\Psi_1\sscdot\ga_{i(1)}^{-1} \,\cap \La,\,\, \Psi_2\sscdot\ga_{i(2)}^{-1}\,\cap \La\,,\ldots\,$ is a F\o lner sequence in $\La$.

 Let $\Si(1)\subset \Si(2)\subset \ldots$ be an ascending sequence of finite subsets of $\La$ such that $\bigcup_{j=1}^\infty \Si(j) = \La$. Note that since $\Sigma(j)\subset \Lambda$, we see that $\partial_{\Sigma(k)}\Phi_j$ is a disjoint union of the sets $\partial_{\Sigma(k)}\big(\Phi_j\cap \Lambda\sscdot\gamma_i\big)$.

For $j\ssin \N_+$ let $k(j)\ssin \N$ be such that
$$
|\partial_{\Si(j)} \Phi_{k(j)}| <\frac{1}{j^2}\sscdot|\Phi_{k(j)}|.
$$
Furthermore, for $j\ssin \N_+$ let 
$$
I(j):=\{i\ssin \N_+\colon |\partial_{\Si(j)}(\Phi_{k(j)}\sscap \La\sscdot \ga_i)| \,\,\ge\,\, \frac1j\sscdot |\Phi_{k(j)}\sscap \La\sscdot \ga_i|\},
$$
and let $B_j := \Phi_{k(j)}\cap\,\,\bigcup\limits_{i\ssin I(j)}  \La\sscdot \ga_i$.

Then we have  $|\partial_{\Si(j)} B_j |\ge \frac1j |B_j|$. Since $|\partial_{\Si(j)} \Phi_{k(j)}| < \frac{1}{j^2}\sscdot |\Phi_{k(j)}|$, it follows that $|B_j|\le \frac{1}{j}|\Phi_{k(j)}|$. Therefore, the sequence $\Psi_1:= \Phi_{k(1)}\setminus B_1, \Psi_2 := \Phi_{k(2)}\setminus B_2,\ldots$ is also a F\o lner sequence in $\Ga$. Furthermore, for all $i\ssin \N_+$ and all finite sets $\Si\subset \La$ there exists $J\ssin\N_+$ such that for $j>J$ we have
\beq\label{eq-unif}
\frac{|\partial_\Si (\Psi_j \sscdot \ga_i^{-1}\,\,\cap\,\, \La)|}{|\Psi_j\sscdot\ga_i^{-1} \,\,\cap \,\,\La|} <\frac{1}{j}.
\eeq
Let us denote
$$
\Psi(j,i) := \Psi_j \sscdot \ga_i^{-1}\,\,\cap\,\, \La.
$$
Now the inequality~\eqref{eq-unif} implies that for any choice of indices $i(1),i(2),\ldots$ such that the sets   $\Psi(1,i(1)), \Psi(2,i(2)),\ldots$ are non-empty we have that  the sequence $\Psi(1,i(1)), \Psi(2,i(2)),\ldots$ is a F\o lner sequence in $\La$. This, together with the fact that the limit~\eqref{eq-def-of-dimension} exists and is independent of the choice of a F\o lner sequence, implies that for every $\eps>0$ there exists $N\ssin \N$ such that for $j>N$ and all $i\ssin \N_+$ such that $\Psi(j,i)$ is non-empty we have 
\beq\label{eq-1}
\left|\,\frac{\dim_\k \ker T_{\Psi(j,i)}}{|\Psi(j,i)|} - \dim_{\k[\La]} \ker T  \,\right| <\eps.
\eeq

On the other hand, for any $j$ we have 
$$
  T_{\Psi_j} \cong  \bigoplus_{i\colon \Psi(j,i) \neq\emptyset} T_{\Psi(j,i)}.
$$
Thus, by the definition of kernel gradient, for every $\eps>0$ and sufficiently large $j$ we have 
\beq\label{eq-2}
\left|\, \dim_{\k[\Ga]}\ker T \,\,-\,\, \frac{\sum\limits_{i\colon \Psi(j,i) \neq\emptyset} \dim_\k \ker T_{\Psi(j,i)}}{\sum\limits_{i\colon \Psi(j,i) \neq\emptyset} |\Psi(j,i)|} \,\right| <\eps
\eeq

Now the lemma follows from the inequalities \eqref{eq-1} and \eqref{eq-2}, the triangle inequality, and the fact that if $a_1,\ldots, a_m$ are non-negative real numbers and $b_1,\ldots, b_m$ are positive real numbers then the fraction $\frac{a_1+\ldots+a_m}{b_1+\ldots+ b_m}$ lies between the smallest and the largest of the fractions $\frac{a_i}{b_i}$.
\end{proof}

The proof of the following lemma is straightforward using suitable block matrices.
\begin{lemma}\label{lm-appendix-addition} 
Let $\k$ be a field and let $\La$ be an amenable group. Then the set 
$$
\{r\in \R\colon \text{for some matrix $T\!$ over $\k[\La]$ we have $\dim_{\k[\La]}\ker T = r$}\}
$$
of real numbers is closed under addition.\qed
\end{lemma}

Given a finite group $\La$ and an enumeration $\la_1,\la_2,\ldots, \la_{|\La|}$ of the elements of $\La$, we denote by $L_\La\colon \La \into \mat(|\La|\sstimes|\La|, \k)$ the map induced by the left regular representation. In other words, for $i,j,k\ssin\{1,\ldots,|\La|\}$ we have that $L_\La(\la_i)$ is a permutation matrix whose $(j,k)$ entry is equal to $1$ if and only if $\la_i\sscdot \la_j = \la_k$. 

For $m,n\ssin\N_+$ we denote with the same symbol the induced map  $L_\La\colon \mat(m \sstimes n,\,\k[\La]) \into \mat( m\sscdot |\La|\sstimes n\sscdot|\La|, \,\k)$. Furthermore, when $\Ga$ is another group, we use the same symbol 
also for the induced map  
$$
L_\La\colon \mat(m\sstimes n, \,\,\k[\La\sstimes \Ga])\,\, \into\,\, \mat(m\sscdot|\La| \sstimes n\sscdot|\La|,\,\, \k[\Ga]).
$$

\begin{lemma}\label{lm-appendix-matrices} 
Let $\Ga$ be an amenable group and let $\La$ be a finite group. For any $T\in \mat (m \times n, \k[\La\times \Ga])$ we have
$$
    \dim_{\k[\La\times \Ga]}\ker T = \frac{1}{|\La|}\dim_{\k[\Ga]} \ker L_\La(T). 
$$
\end{lemma}
\begin{proof}[Sketch of Proof]
For simplicity we assume $m=n=1$. Let $\Phi\subset \Ga$ and let $\la_1,\la_2,\ldots, \la_{|\La|}$ be the enumeration of the elements of $\La$ used to define $L_\La$. Let 
$$
J\colon \k[\Phi \sstimes \La] \to \k[\Phi]^{|\La|}
$$
be defined by demanding that $J(\phi, \la_j)$ is the vector whose $j$-th coordinate is $\phi$ and all the other coordinates are $0$. It is straightforward to check that $J$ is a $\k$-linear isomorphism which intertwines $T_{\Phi\sstimes \La}$ and the compression of  $L_\La(T)$ to $\k[\Phi]^{|\La|}$. The lemma follows by noting that if $\Phi_1,\Phi_2,\ldots$ is a F\o lner sequence for $\Ga$ then $\Phi_1\sstimes \La, \Phi_2\sstimes \La,\ldots$ is a F\o lner sequence for $\Ga\sstimes \La$. 
\end{proof}

Let us recall that $\ZZ_2 = \langle u \rangle$  is the cyclic group of order $2$.

\begin{lemma}\label{lm-appendix-multiplication}
 Let $\k$ be a field such that $\chr(\k)\neq 2$, let $\Ga$ be an amenable group, let $m,n\ssin \N$ be such that  $m\ge n$, let  $T\ssin \mat( m\sstimes n, \k[\Ga])$, and let $I\ssin \mat(m\sstimes n, \k)$ be such that $\ker (I) =\{0\}$. Let $S \ssin \mat( m\sstimes n, \k[\Ga\sstimes \ZZ_2]) $ be defined as   
$$
S := I\cdot \diag_n (\frac{1-u}{2}) + T \cdot \diag_n(\frac{1+u}{2}).
$$
Then 
 $$
    \dim_{\k[\Ga\sstimes \ZZ_2]} \ker S  = \frac{1}{2}\dim_{\k[\Ga]} \ker T.  
 $$
\end{lemma}

\begin{proof}[Sketch of Proof]
Let $\Phi_1,\Phi_2,\ldots$ be a F\o lner sequence for $\Ga$. Clearly  $\Phi_1 \sstimes \ZZ_2,\, \Phi_2\sstimes \ZZ_2,\ldots $ is a F{\o}lner sequence for $\Ga \sstimes \ZZ_2$. Let us fix $i$ and set $\Phi \,{:=}\, \Phi_i$. Let $V_- := \lin (\frac{1-u}2\sscdot\phi\colon \phi \ssin \Phi)$ and  $V_+ := \lin (\frac{1+u}2\sscdot\phi\colon \phi \ssin \Phi)$. 

It is straightforward to check that $\k[\Phi\times \ZZ_2]^m = V_-^m \oplus V_+^m$ and that both $V_-^m$ and $V_+^m$ are $S_{\Phi\sstimes \ZZ_2}$-invariant. Furthermore,  the restriction of $S_{\Phi\sstimes \ZZ_2}$ to $V_-^m$ is isomorphic to $I_\Phi$ and the restriction to $V_+^m$ is isomorphic to $T_\Phi$. Since we also have $\dim V_+ = \dim V_-$, the lemma follows. 
\end{proof}

\subsection{Graphs}\mbox{}

Let us state our graph-theoretic conventions. A \textit{directed graph} is a pair $(V,E)$, where $V$ is a set and $E$ is a subset of $V\sstimes V$. In particular, each vertex of a directed graph is allowed to have a single self-loop and there are no multiple edges.

Let $L$ be a set. A \textit{directed graph with edges labeled by $L$ } is a pair $(V,E)$, where $V$ is a set and $E$ is a subset of $V\sstimes V \sstimes L$, such that for any $(v,w)\ssin V\sstimes V$ there is at most one $s\ssin L$ such that $(v,w,s)\ssin E$. 

A \textit{directed multigraph with edges labeled by $L$} is a pair $(V,E)$, where $V$ is a set and $E$ is a subset of $V\sstimes V \sstimes L$.

If $G = (V,E)$ is a directed graph with edges labeled by $L$, and $L$ is a subset of a ring $R$, then the \textit{adjacency operator} on $G$ is the unique map $R$-linear map $M\colon R[V] \to R[V]$ such that if  $(v,w,s)\ssin E$ then the coefficient of $M(v)$ at $w$ is equal to $s$.

\subsection{Pontryagin duality and ring homomorphisms}\mbox{}

Let $A$ be a group isomorphic to the direct sum of infinitely many copies of $\ZZ_2$,  let $X := \wh A$ be the Pontryagin dual of $A$ , and let 
$\mu$ be the Haar measure on $X$ normalized so that $\mu(X) = 1$. The key property of $A$ is that all the homomorphisms $A \to U(1)$ factor through $\{-1,1\}$. From now on we will always assume that $\k$ is a field with $\chr(\k)\neq 2$. As such, we obtain an embedding $\k[A]\into \fun(X, \k)$, where $\fun(X,\k)$ denotes the set of (set theoretic) maps from $X$ to $\k$. The image of $p\in \k[A]$ in $\fun(X,\k)$ is denoted by $\wh p$. This
embedding commutes with base field changes in the sense that we will now explain.

Following the notation of \cite[Example 3 after Corollary
3.2]{atiyah-macdonald-introduction-to-commutative-algebra}, we let $\Z_2$ be the
ring of rational numbers which can be written with denominator which is a power of
$2$. Let $\rho\colon \Z_2 \to \k$ be the natural homomorphism. We use the same
letter $\rho$ for the induced homomorphisms $\fun(X,\Z_2) \to \fun(X,\k)$ and $\Z_2[F] \to \k[F]$ for any set $F$. 

We say that $p\ssin \Q[A]$ is a \textit{projection} if the range of $\wh p\ssin \fun(X,\Q)$ is a subset of $\{0,1\}$. If $p,r\ssin \Q[A]$ are projections then we will write $p\prec r$ if and only if $\supp(\wh p) \subset \supp(\wh r)$. We will say that $p$ and $r$ are \textit{orthogonal}, denoted by $p\perp r$, if $\supp(\wh p)\cap \supp(\wh r) = \emptyset$.

\begin{lemma}\label{le-comm}
\begin{enumerate}
\item 
The following diagram commutes:
$$
\begin{CD}
\Z_2[A] @>\wh{}>> \fun(X, \Z_2)\\
@V\rho VV            @V\rho VV\\
\k[A] @>\wh{}>> \fun (X, \k)
\end{CD}
$$
\item Let $p\in \Z_2[A] \setminus \{0\}$ be a non-zero projection. Then $ \supp (\rho(\wh p)) = \supp(\wh p)$ and $\rho(p) \neq 0$.
\end{enumerate}
\end{lemma}
\begin{proof}
The first item is a straightforward exercise in Pontryagin duality. 


Since the Pontryagin duality is a ring homomorphism, we have $\wh p\,{}^2 = \wh p$, and hence for every $x\ssin X$ we have $\wh p(x) \ssin \{0,1\}$. Since $\rho(0) = 0$ and $\rho(1) = 1$ we deduce that 
 $ \supp (\rho(\wh p)) = \supp(\wh p)$. Noting that the Pontryagin duality map $\Z_2[A] \to \fun(X,\Z_2)$ is an embedding, we conclude that $\supp(\wh p) \neq \emptyset$ and hence also $\supp (\rho(\wh p)) \neq \emptyset$. By the first item, this shows $\rho(p)\neq 0$ and finishes the proof.
\end{proof}

\begin{lemma}\label{lem-basis}
Let $F$ be a finite subgroup of $A$ and let $p\in \Z_2[F]$ be a projection. Then
\beq\label{eq-dims}
\dim_\k \lin \,(\rho(p)a\colon a{\in}F) = |F|{\cdot}\mu(\supp(\wh p)).
\eeq
\end{lemma}
\begin{proof}
Let $X_F$ be the Pontryagin dual of $F$, let $\nu$ be the Haar measure on $X_F$ normalized so that $\nu(X_F)=1$ and let $\wh p_F\colon X_F \to \{0,1\}$ be the Pontryagin dual of $p$ as an element of the group ring $\Z_2[F]$. A simple exercise in Pontryagin duality shows that $\nu(\supp(\wh p_F)) = \mu(\supp(\wh p))$. Thus it is enough to show that
\beq\label{eq-dims2}
\dim_\k \lin \,(\rho(p)a\colon a{\in}F) = |F|{\cdot}\nu(\supp(\wh p_F)).
\eeq

Let $M\colon \k[F] \to \k[F]$ and $N\colon \fun(X_F, \k) \to \fun(X_F, \k)$ be the multiplications by $\rho(p)$ and $\rho(\wh p_F)$, respectively.  The left-hand side of \eqref{eq-dims2} is equal to $\dim_\k \im( M)$ and the right-hand side is equal to $\dim_\k \im (N)$. They are equal because the Pontryagin duality map $\k[F] \to \fun(X_F, \k)$ is a $\k$-algebra isomorphism which intertwines $M$ and $N$.
\end{proof}

\subsection{$T$-graphs}\mbox{}

Let $A$, $X$ and $\mu$ be as in the previous section, and let $\Ga\actson A$ be an action of a countable group $\Ga$ on $A$ by group automorphisms. We recall that the outcome of the action of $\ga\ssin \Ga$ on $a\ssin A$ is denoted by $\ga.a$. The central dot
$\cdot$ is reserved for group and ring multiplications, and ``the implied dot is the central one''; for example, if $\ga\ssin \Ga$ and $a\ssin A$ then $\ga a$ should be read as $\ga\sscdot a$.

The dual action $\Ga\,{\actson}\, X$ is the unique continuous action for which the diagram in Lemma~\ref{le-comm} is $\Ga$-equivariant. It is easy to check that the action $\Ga\actson X$ is measure-preserving.

Let  
\begin{equation*}
T := \sum_{i=1}^m \ga_i f_i \in \Z_2[\Ga\ltimes A]\qquad\ga_i \in \Ga,\; f_i\in \Z_2[A].
\end{equation*}
A \textit{$T$-graph} is a finite set $G\subset \Z_2[A]$ with the following properties.

\begin{enumerate}
\item The elements of $G$ are pairwise orthogonal non-zero projections.
\item For all $p \ssin G$ and $k\ssin [m]$ we have either $\supp (\wh p) {\,\subset\,} \supp (\wh f_k)$ or  $$\supp (\wh p) {\,\cap\,} \supp (\wh f_k) = \emptyset.$$
\item If for some $p \ssin G$ and $k\ssin[m]$ we have $\supp (\wh p)  {\,\subset\,} \supp (\wh f_k)$ then $\ga_k. p \in G$ and $\wh f_k$ is constant on $\supp (\wh p)$.
\end{enumerate}

\begin{remarks} 
\begin{enumerate}[(i),wide,itemsep=5pt]
\item A careful reader will note  a slight ambiguity in the above definition: whether or not a set $G\subset \Z_2[A]$ is a $T$-graph might depend on the choice of a decomposition of $T$ as a sum $T = \sum_{i=1}^m \ga_i f_i$. Thus the expression "$T$-graph" should be replaced by ``$T$-graph with respect to the decomposition $T = \sum_{i=1}^m \ga_i f_i$.''   

\vspace{3pt} However, this this will never be an issue for us, because we always work with exactly one representation of $T$ as a sum $\sum_{i=1}^{m} \ga_i f_i$. 

\item Let $S,T \in \Z_2[\Ga\ltimes A]$. We note that it can easily happen that a finite set $G\subset \Z_2[A]$ is both an $S$-graph and a $T$-graph.
\end{enumerate}
\end{remarks}

Given a $T$-graph $G$, let us define a directed multigraph $\cal R(G, T)$ with edges labeled by the set $\{\ga_1,\ldots, \ga_m\}$. We let the vertex set of $\cal R(G,T)$ be equal to the set $G$, and we let $(p,q,\ga) \in G\times G\times \{\ga_1,\ldots, \ga_m\}$ be an edge if and only if for some $k\ssin [m]$ we have $\ga=\ga_k$,  $\supp (\wh p) \subset \supp (\wh f_k)$ and $\ga_k.p = q$.  

We note in particular that at each vertex of $\cal R(G,T)$ the out-edges are uniquely labeled, and so the set of edges of $\cal R(G,T)$ can be naturally identified with a subset of $G\times \{\ga_1,\ldots, \ga_m\}$, by sending an edge $(p,q,\ga)$ to $(p,\ga)$. 

A \textit{path} in $\cal R(G,T)$ is, speaking informally, a finite sequence of
edges in $\cal R(G,T)$ such that any two consecutive edges are adjacent. More
precisely, a path from $p_1$ to $\lambda_w.p_w$ is a sequence of elements
$(p_1,\la_1),\ldots, (p_w,\la_w)$ of $G\sstimes \Ga$ such that for all
$i=1,\ldots, w{-1}$ we have $\la_i.p_i = p_{i+1}$ and for all $i=1,\ldots, w$
we have that either $(p_i,\la_i)$ or $(\la_i.p_i, \la_i^{-1})$ is an edge.

We say that $G$ is \textit{connected} if for any two vertices in $\cal R(G,T)$ there exists a path which connects them.

 We define the \textit{label} of such a path to be the product $\la_{w}\sscdot\ldots\sscdot \la_{1} \in \Ga$. It is clear that if $p,q\in G$ and $\ga$ is a label of a path from $p$ to $q$ then $\ga.p=q$.

 A \textit{loop} is a path whose starting
and final vertices are equal.  For $p\ssin G$ we define 
$$
\Ga(p) := \{\ga\in \Ga\colon \ga \text{ is a label of a path in $\cal R(G,T)$ starting at $p$}\}\cup \{e\},
$$
where $e$ is the neutral element of $\Ga$. 

We say that $G$ is \textit{simply-connected} if it is connected and the label of any loop in $\cal R(G,T)$ is equal to the neutral element of $\Ga$. If $G$ is simply-connected then in $\cal R(G,T)$ the label of any self-loop is equal to the identity and there are no multiple edges. Furthermore, simply connected $T$-graphs have the following basic properties which we will implicitly use.

\begin{lemma}
  \label{lem:compare_}
  Let $G$ be a simply-connected $T$-graph, let $p\in G$, and let
  $\la\ssin\Gamma(p)$. 
\begin{enumerate}[(i),nosep]
\item The map $\alpha\colon \Gamma(p)\to G$ which sends $\gamma\ssin \Gamma(p)$ to $\ga.p$ is a
  bijection. 

\item We have $\Gamma(\la.p)=\Gamma(p)\cdot
  \la^{-1}:=\{\gamma\cdot\la^{-1}\colon\,\, \gamma\in \Gamma(p)\}$.
\end{enumerate}
\end{lemma}
\begin{proof}[Sketch of Proof]
  (i) By definition of ``connected'', the map $\alpha$ is surjective. If
  $\gamma_1.p=\gamma_2.p=q$ for $\gamma_1,\gamma_2\in \Gamma(p)$, then the
  concatenation of a path from $p$ to $q$ with label $\gamma_1$ and the
  inverse of a path from $p$ to $q$ with label $\gamma_2$ is a loop with label
  $\gamma_2^{-1}\gamma_1$. Therefore $\gamma_2=\gamma_1$ and we see that
  $\alpha$ is injective.

  (ii) Every path starting at $\la.p$ with a label $\ga\in\Gamma(\la.p)$ can be
  pre-concatenated with a path from $p$ to $\la.p$, with label $\la$,
  resulting in a path starting at $p$, and with label $\ga\cdot \la$. Therefore,
  $\Gamma(\la.p)\cdot \la\subset \Gamma(p)$. As both sets are finite and of
the  same cardinality, they are equal.
\end{proof}

 If $G$ is simply-connected and $p,q\in G$ then we define $\ga(p,q)$ to be the unique element of $\Ga$ which appears as the label of a path from $p$ to $q$. In particular we have $\ga(p,q).p=q$. 

Let $G$ be a $T$-graph, let $p,q \in G$, and let $\k$ be a field with $\chr(\k)\neq 2$.  We define $M(p,q;\k)\ssin \k$  by setting 
$$
    M(p,q;\k) := \sum_{i\colon \ga_i.p=q} \rho(\langle f_i,p\rangle)
$$
where $\langle f_i, p\rangle\in \Z_2$ is defined to be the unique value of $f_i$ on $\supp (\wh p)$. We use the convention that the empty sum is equal to $0$, so if there is no edge between $p$ and $q$ in $\cal R(G,T)$ then $M(p,q;\k) = 0$. We let $M(G,T;\k)\colon \k[G]\to \k[G]$ be the  $\k$-linear map such that for $p\in G$ we have  
$$
    M(G,T;\k) (p) := \sum_{q\in G} M(p,q;\k) q.
$$

For the applications in Section~\ref{sec-app} it is convenient to define a directed graph $\cal S(G,T;\k)$ with edges labeled by the elements of $\k$. We let the vertex set of $\cal S(G,T;\k)$ be equal to $G$ and we let $(p,q,s)$ be an edge of $\cal S(G,T;\k)$ if for some $\ga\ssin \{\ga_1,\ldots, \ga_k\}$ the triple $(p,q,\ga)$ is an edge in $\cal R(G,T)$ and $s = M(p,q;\k)$. Let us note that for any field $\k$ the adjacency operator on $\cal S(G,T;\k)$ is equal to $M(G,T;\k)$.

We let $M(p,q) := M(p,q;\Q)$, $M(T,G) := M(T,G;\Q)$ and $S(G,T) := S(G,T;\Q)$. Clearly we have $M(p,q;\k) = \rho(M(p,q))$. 

The following is the key observation about simply-connected $T$-graphs. 

\newcommand{\rI}{J}
\begin{lemma}\label{lem-ko}
Let $G$ be a simply-connected $T$-graph, let $q\ssin G$, let $a\ssin A$, and let $q'\ssin \Z_2[A]$ be a projection such that $q'\prec q$. Let $W$ be the $\k$-linear subspace of $\k[\Ga\ltimes A]$ spanned by the elements $ \ga(q,r)\sscdot \rho(q')\sscdot a$, where $r\in G$. Then $W$ is $\rho(T)$-invariant and isomorphic to $\k[G]$ via an isomorphism which intertwines $\rho(T)$ and $M(G,T;\k)$. 
\end{lemma}
\begin{proof}
Let us first check that the elements  $\ga(q,r)\sscdot \rho(q')\sscdot a \in \k[\Ga\ltimes A]$,  where $r\ssin G$, are linearly independent. For this let us suppose that for some scalars  $s(r)\ssin \k$, where $r\ssin G$, we have
$$
\sum_{r\ssin G} s(r) \ga(q,r)\sscdot \rho(q')\sscdot a  =0, 
$$
and hence also
$$
\sum_{r\ssin G} s(r) \ga(p,q)\sscdot \rho(q') =0.
$$
Since the elements $\ga(q,r)$ of the group $\Ga$, where $r\ssin G$, are pairwise distinct and $\rho(q')\ssin \k[A]$,  we see that in order to deduce that $s(r)\,{=}\, 0$ for all $q\ssin G$, it is enough to check that $\rho(q') \neq 0$, which is part of Lemma~\ref{le-comm}. 

Thus we can define a $\k$-linear isomorphism $\rI\colon W \to \k[G]$ by  setting 
$$
\rI(\ga(q,r)\sscdot \rho(q')\sscdot a):= r.
$$ 

It remains to check that $\rI$ intertwines  the restriction of $\rho(T)$ to $W$ and $M(G,T;\k)$. Let us fix $r\ssin G$ and let 
$$
r' := \ga(q,r)\sscdot q'\sscdot \ga(q,r)^{-1}.
$$
Note that $r'\ssin \Z_2[A]$ is a projection such that $r' \prec r$. Now we have
\begin{multline*}
\rho(T)(\ga(q,r)\sscdot \rho(q')\sscdot a) =\rho(T)(\rho(r')\sscdot \ga(q,r)\sscdot a)= \sum_{i=1}^m \ga_i \rho(f_i)\cdot \rho(r')\sscdot \ga(q,r)\sscdot a.
\end{multline*}

Note that $\rho(f_i)\sscdot \rho(r')$ is equal to the reduction mod $\chr(\k)$ of the unique value of $\wh{f}_i$ on $\supp(\wh r)$. It follows that the right hand side above is equal to
$$
\left(\sum_{s \in G} \ga(r,s) \sscdot {M(r,s;\k)}\sscdot \rho(r')\right) \cdot \ga(q,r)\sscdot a = \sum_{s \in G} {M(r,s;\k)}\sscdot  \ga(q,s)\sscdot \rho(q')\sscdot a.
$$

On the other hand, from the definition of the map $M(G,T;\k)\colon \k[G] \to \k[G]$ we have 
$$
M(G,T;\k)(q) = \sum_{s\in G} {M(q,s;\k)}\sscdot s.
$$ 
This establishes that $\rI$ intertwines $\rho(T)$ and $M(G,T;\k)$ and finishes the proof.
\end{proof}

\begin{lemma}\label{lem-refin}
Let $G\subset \Z_2[A]$ be a finite set of pairwise orthogonal projections, and
let $\Phi$ be a finite subset of $\Ga$. There exists a finite set $K\subset
\Z_2[A]$ of pairwise orthogonal projections with sum equal to $1$, i.e.~such that $\bigcup_{p\ssin K} \supp(\wh p) =X$ and for all $p\ssin K$, $\phi\ssin \Phi$ and $q\ssin G$ we have either $\phi.p\prec q$ or $\phi.p\perp q$.
\end{lemma}

\begin{proof}
We prove the lemma by induction on the cardinality of the set $\Phi$. If $\Phi$ is empty then the desired statement is true since we can take $K:=
\{1\}$. Thus let us assume that $\Phi = \Phi_0 \cup \{\psi\}$ where $|\Phi_0|<|\Phi|$, and let $K_0 \subset \Z_2[A]$ be a finite set guaranteed by the inductive assumption, i.e~for all $p\ssin K_0$, $\phi\ssin \Phi_0$ and $q\ssin G$ we have either $\phi.p\prec q$ or $\phi.p\perp q$.

Let $G^+ :=  G \cup \{1-\sum_{q\ssin G} q\}$. We define $K$ to be the set
$$
 K: = \bigcup_{q\ssin G^+}\bigcup_{p\ssin K_0} \{p\cdot \phi^{-1}.q\colon \phi\ssin \Phi \}.
$$
It is straightforward to check that $K$ has the desired properties.
\end{proof}

\section{Computational tool}\label{sec-ct}
Let $T_1,\ldots, T_l\in \Z_2[\Ga \ltimes A]$. If a set $G\subset \Z_2[A]$ has the property that for all $i\ssin [l]$ we have that $G$ is a $T_i$-graph, then we say that $G$ is a $(T_1,\ldots, T_l)$-graph. Two $(T_1,\ldots, T_l)$-graphs $G_1$ and $G_2$ are \textit{orthogonal} if for all $p\ssin G_1$, $q\ssin G_2$ we have $p\perp q$. We say that a sequence $G_1,G_2,\ldots $ of $(T_1,\ldots, T_l)$-graphs is \textit{exhausting} if $G_i$'s are pairwise orthogonal and $\sum_{i=1}^\infty \mu(\supp(G_i)) =1$.  

We are ready to state and prove the analog of Proposition~\ref{prop-tool-free} for kernel gradients.

\begin{theorem}\label{thm:compu-tational-tool}
Let $\k$ be a field with $\chr(\k)\neq 2$, let $\Ga$ be a countable amenable group, let $A$ be a countable group isomorphic to a direct sum of copies of $\ZZ_2$, and let $\Ga\actson A$ be an action by group automorphisms.

Let $T_1,\ldots, T_l\in \Z_2[\Ga \ltimes A]$ and let $\,G_1,G_2,\ldots\, $ be an exhausting sequence of simply-connected $(T_1,\ldots, T_l)$-graphs. Then we have 
\begin{equation}\label{eq-l}
    \dim_{\k[\Ga \ltimes A]} \ker\left( 
\left[\begin{array}{c}
\rho(T_1)\\
\vdots\\
\rho(T_l)
\end{array}\right]\right)
 = 
\sum_{i=1}^\infty \frac{\mu (\supp (G_i) )}{|G_i|} \dim_\k \ker \left(
\left[\begin{array}{c}
M(G_i,T_1;\k)\\
\vdots\\
M(G_i,T_l;\k)
\end{array}\right]\right).
\end{equation}
\end{theorem}
\vspace{5pt}
Before we start the proof, let us make two remarks.

\vspace{5pt}
\begin{remarks}\label{rem-pu}
\begin{enumerate}[(i),wide,itemsep=5pt]
\item The way Theorem~\ref{thm:compu-tational-tool} will be applied in Section~\ref{sec-app}  to prove Theorems \ref{thm:intro-finitely-generated}, \ref{thm:intro-lamplighter3} and \ref{thm-intro-positive} is essentially as follows. For each of the theorems we will take a specific element $T$ (or several such elements)  together with an explicit exhausting sequence $G_1,G_2,\ldots\,$ of simply-connected $T$-graphs from \cite{grabowski-on-turing-dynamical-systems-and-the-atiyah-problem} or \cite{arXiv:grabowski-2010-2}. 

\vspace{3pt}
The measures $\mu(\supp (G_i))$ can  be copied from
either~\cite{grabowski-on-turing-dynamical-systems-and-the-atiyah-problem}
or~\cite{arXiv:grabowski-2010-2}, and so in order to compute the right hand
side of~\eqref{eq-l}, it is enough to compute the dimensions $\dim_\k\ker
M(G_i,T; \k)$ of the adjacency operators on the labeled graphs $\cal
S(G_i,T;\k)$. This is an exercise in linear algebra, because the graphs $\cal
S(G_i,T;\Q)$ are explicitely described in
\cite{grabowski-on-turing-dynamical-systems-and-the-atiyah-problem} and
\cite{arXiv:grabowski-2010-2}, their edge labels are all elements of $\Z$, and the graphs $S(G_i,T;\k)$ arise from the graphs $S(G_i,T;\Q)$ by reducing the labels modulo $\chr(\k)$.

\vspace{3pt}
Thus Theorem~\ref{thm:compu-tational-tool} allows us to compute the left hand side of~\eqref{eq-l}, and then Proposition~\ref{prop-eck} allows us to deduce the existence of the CW-complexes with suitable homology gradients.

\item 

Let $\La$ be a residually finite group. Then for any $p\ssin\N_+$ the group $\ZZ_2^p\wr \La$ also is residually finite, and thus we can take a sequence $\De_1\supset \De_2\supset \ldots$ of finite-index normal subgroups of $\ZZ_2^p\wr\La$ such that $\bigcap_{i=1}^\infty \De_i = \{e\}$.

\vspace{3pt}
Let $T\ssin \k[\ZZ_2^p\wr \La]$ and let $T_i$ be the image of $T$ in   $\k[(\ZZ_2^p\wr\La)\,/\,\De_i]$. By the results of~\cite{MR2823074}, if $\La$ is amenable then
\beq\label{eq-rep}
\dim_{\k[\ZZ_2^p\wr \La]}\ker T  \,=\, \lim_{i\to \infty} \frac{\dim_\k\ker \rho(T_i)}{[\De:\De_i]}.
\eeq
In particular, the limit on the right hand side exists. On the other hand, if $\La$ is not amenable then in general it is not known if the limit on the right hand side of~\eqref{eq-rep} exists. 

\vspace{3pt}
Let us see how Theorem~\ref{thm:compu-tational-tool} provides a class of ``test cases'' for the question whether the limit on the right hand side of~\eqref{eq-rep} always exists. 
Let $\ga_1,\ldots, \ga_n\ssin \La$, let $U\ssin \Z[\ZZ_2^p]$ be the sum of all the elements of $\ZZ_2^p$, and let 
$$
T:= \sum_{i=1}^n \ga_i \frac{U}{2^p} + \frac{U}{2^p}\ga_i^{-1} \,\,\in\,\, \Z_2[\ZZ_2^p\wr \La].
$$ 
If $p$ is large enough then either \cite{Lehner_Neuhauser_Woess:On_the_spectrum_of} or \cite[Subsection 3.2]{grabowski-on-turing-dynamical-systems-and-the-atiyah-problem} shows that  there exists an exhausting sequence of simply-connected $T$-graphs. As such, we can compute the right hand side of \eqref{eq-l}, and in view of the equalities~\eqref{eq-l} and~\eqref{eq-rep} which are valid when $\La$ is amenable, it would be interesting to determine whether the right hand side of~\eqref{eq-rep} exists and is equal to the right hand side of~\eqref{eq-l}, for all choices of the sequence $\De_i$.
\end{enumerate}
\end{remarks}
\vspace{5pt}

\begin{proof}[Proof of Therem \ref{thm:compu-tational-tool}]

We give the proof just for one operator $T:= T_1 = \sum_{i=1}^m \ga_i f_i$, where $\ga_i \ssin \Ga$ and $f_i\ssin \Z_2[A]$. The general version does not present any additional difficulties, except for requiring more involved notation.

Let $\eps \ssin (0,1)$ and let $N$ be such that $\sum_{i=1}^N \mu(\supp(G_i)) >1\ssminus \eps$. Let $G := \bigcup_{i=1}^N G_i$ and let $\Si:=\bigcup_{q \in G} \Ga(q)$.

Let $\Phi\subset \Ga$ be any finite subset such that $|\partial_\Si(\Phi)| <\eps |\Phi|$ and let $\wb{ \Phi} := \Phi \cup \partial_\Si \Phi$.  By Lemma~\ref{lem-refin} there exists a finite set $K\subset\Z_2[A]$ of pairwise orthogonal projections such that we have the following two properties:

\begin{enumerate}[(i),nosep]
\item $\bigcup_{p\ssin K} \supp(\wh p) = X$, and 
\item  for all $p\ssin K$, $\phi\ssin \Phi$ and $q\ssin G$ we have either $\phi.p\prec q$ or $\phi.p \perp q$. 
\end{enumerate}

Let $F\subset A $ be a finite subgroup of $A$ such that $p\ssin \Z_2[F]$ for all $p\ssin K$. Let $T_{\Phi\sscdot F}$ be the compression of $\rho(T)$ to the subspace $\k[\Phi\sscdot F]$ of $\k[\Ga\ltimes A]$. Since finite subgroups of $A$ can be arranged into a F\o lner sequence  for $A$, by Lemma~\ref{lem-semi} it is enough to show that
\beq\label{eq-todo} 
	\left|\dim_\k\ker T_{\Phi\sscdot F} - \sum_{i=1}^N \frac{\mu (\supp (G_i) )}{|G_i|} |\Phi||F| \cdot \dim_\k \ker M(G_i,T;\k)\right| \le 3\eps\sscdot|\Phi||F|.
\eeq
This will occupy the rest of the proof.

Let us fix for each $p\ssin K$ a set $F(p)\subset F$ such that the elements $\rho(p)a$, $a\in F(p)$, form a basis of the subspace of $\k[A]$ spanned by the elements $\rho(p)a$, $a\in F$. Lemma~\ref{lem-basis} shows that 
\beq\label{eq-no}
|F(p)| = \mu(\supp(\wh p))\cdot |F|.
\eeq

Let $(\phi,p)\ssin \Phi\sstimes K$. We say that the pair  $(\phi,p)$ is \textit{lovely} if there is $q\ssin G$ such that $\phi.p  \prec q$. We say that a triple $(\phi,p,a)\ssin \Phi\sstimes K \sstimes F$ is \textit{lovely} if $(\phi,p)$ is lovely and $a\ssin F(p)$.  

If $p\ssin \Z_2[A]$ is a projection such that for some $q\ssin G$ we have $p\,{\prec}\, q$ then we define $Q(p) := q$ and $G(p) = G_i$, where $i$ is such that $Q(p)\ssin G_i$. In particular if $(\phi,p)$ is lovely then $Q(\phi.p)$ and $G(\phi.p)$ are well-defined.

If $(\phi,p,a)$ is a lovely triple, we define
$$ 
 Y(\phi, p, a) := \lin \big(  \ga \cdot \phi\cdot \rho(p) \cdot a \colon \,\, \ga\ssin \Ga(Q(\phi.p))\,\big) \,\subset \,\k[\wb\Phi\sscdot F].
$$ 

The following lemma follows directly from Lemma~\ref{lem-ko}.

\begin{lemma}\label{lem-act}
Let $(\phi,p,a)$ be a lovely triple. Then the space $Y(\phi, p, a)$ is $\rho(T)$-invariant and there exists a $\k$-linear isomorphism $Y(\phi,p,a) \to \k[G(\phi.p)]$ which intertwines $\rho(T)$ and $M(G(\phi.p),T;\k)$.\qed
\end{lemma}

The next lemma lists two important properties of lovely triples.

\begin{lemma}\label{lm-props-of-Y} 
\begin{enumerate}[(1),wide,itemsep=5pt]

\item Let $(\phi_1, p_1,a_1)$ and $(\phi_2, p_2,a_2)$ be two lovely triples. Then $Y(\phi_1, p_1,a_1) = Y(\phi_2, p_2,a_2)$ if and only if  $a_1 {\,=\,} a_2$, $p_1 {\,=\,}p_2$  and for some $\ga\ssin \Ga(Q(\phi_1.p_1))$ we have $\ga\phi_1 = \phi_2$.
\item Let $V$ be the subspace of $\k[\wb \Phi {\cdot} F]$ generated by the spaces $Y(\phi,p,a)$, where $(\phi, p,a)$ runs through all lovely triples. Then $V$ is a direct sum of all the different spaces $Y(\phi,p,a)$.

\end{enumerate}

\end{lemma}

\begin{proof}

The "$\isimplied$" direction of (1) is easy to check. Let us now prove (2) and the other direction of (1). Let $I\ssin \N_+$ and let $(\phi_1,p_1,a_1), (\phi_2,p_2,a_2), \ldots\,,(\phi_I,p_I,a_I)$ be a sequence of lovely triples with the following property: for all distinct $i,j\ssin I$ and all $\ga\ssin \Ga(Q(\phi_i.p_i))$  we have $(\ga\sscdot \phi_i,p_i,a_i) \neq  (\phi_j,p_j,a_j)$.

For $i\ssin [I]$ let $v_i\ssin Y(\phi_i,p_i,a_i)$ be such that $\sum_{i=1}^I v_i = 0$. In order to finish the proof we need to show that for all $i\ssin [I]$ we have $v_i=0$.

Let us start by writing each $v_i$ in the standard basis:
\beq\label{eq-deco}
    v_i = \sum_{\ga\ssin  \Ga(Q(\phi_i.p_i))} s(i,\ga) \sscdot \ga \cdot \phi_i \cdot \rho(p_i) \cdot a_i, 
\eeq
where $s(i,\ga)\ssin \k$. Thus we have 
\beq\label{eq-summ}
\sum_{i=1}^I\,\, \sum_{\ga\ssin \Ga(Q(\phi_i.p_i))} \,\,s(i,\ga) \cdot \ga \cdot \phi_i \cdot \rho(p_i) \cdot a_i =0.
\eeq

Let us fix $\psi \ssin \wb\Phi$, $p\ssin G$ and $a\ssin F(p)$.  Let $A\subset [I]$ be the set of those $i$ for which $p_i = p$, $a_i =a$ and for some $\ga\ssin  \Ga(Q(\phi_i.p_i))$ we have $\ga\sscdot \phi_i = \psi$. Clearly such $\ga$ is unique and we denote it by $\ga_i$. Let us argue that
\beq\label{rns}
\sum_{i\ssin A} s(i,\ga_i) \cdot \ga_i \cdot \phi_i \cdot \rho(p_i) \cdot a_i =0.
\eeq
Indeed, it is clear that in the sum~\eqref{eq-summ} we may take all those summands for which $\ga\sscdot\phi_i = \psi$ and still obtain $0$. In the resulting sum we may take only those summand for which $p_i= p$ and still obtain $0$, since for $r\ssin G$ distinct from $p$ we have $p\perp r$.  And in that sum we may take only those summands for which $a_i =a$ and still obtain $0$ by the definition of $F(p)$. Thus we see that \eqref{rns} holds.

Since $\psi$, $p$ and $a$ are arbitrary, it is enough to show that for all $i\ssin A$ we have that $s(i,\ga_i) = 0$.

Since for all $i\ssin A$ we have $\ga_i\sscdot \phi_i = \psi$, we also have $\ga_i\sscdot \phi_i.p = \psi.p$. It follows that for all $i\ssin A$ we have $\ga_i.Q(\phi_i.p) = Q(\psi.p)$.  Hence, for all $i\ssin A$ we have $G(\phi_i.p) = G(\psi.p)$ and $\ga_i^{-1} \ssin \Ga(Q(\psi.p))$. Thus for all $i,j\ssin A$ we have $\ga_i^{-1}\sscdot \ga_j \ssin \Ga(Q(\phi_j.p_j))$. 

But clearly $\ga_i^{-1}\sscdot \ga_j \sscdot \phi_j = \phi_i$ so by our initial  assumption on the sequence $(\phi_1,p_1,a_1)$, $(\phi_2,p_2,a_2),$ $\ldots\,,(\phi_I,p_I,a_I)$, we see that in fact $|A|=1$ and hence $s(i,\ga_i)=0$ for all $i\ssin A$, as needed. This concludes the proof.
\end{proof}

We proceed to show that the dimension of the space $V$ defined in the previous lemma is large.

\begin{lemma}\label{lem-scount}
For every $i\ssin[N]$ the set of lovely triples $(\phi,p,a)\ssin \Phi\sstimes K\sstimes F$ with the property $G(\phi.p) = G_i$ has cardinality  
$$
\mu(\supp(G_i))\sscdot |F||\Phi|. 
$$
\end{lemma}
\begin{proof} Indeed, let us fix $i\ssin =N$ and $\phi\ssin \Phi$. Let $L\subset K$ be the set of those $p$ for which $\phi.p \prec q$ for some $q\ssin G_i$. Since the action $\Ga\ssactson X$ is measure-preserving, we have $\sum_{p\ssin L} \mu(\supp(\wh p)) = \mu(\supp(G_i)$, and now the claim follows from \eqref{eq-no}.
\end{proof}

For each $i\ssin [N]$ let $V_i$ be a maximal set of lovely triples with the following two properties:
\begin{enumerate}[(i),nosep]
\item for all $(\phi,p,a)\ssin V_i$ we have $G(\phi.p) = G_i$, and 
\item for distinct lovely triples $(\phi_1,p_1,a_1)$ and $(\phi_2,p_2,a_2)$ in $V_i$ we have that the spaces $Y(\phi_1,p_1,a_1)$ and $Y(\phi_2,p_2,a_2)$ intersect trivially.
\end{enumerate}

Now Lemmas~\ref{lm-props-of-Y} and~\ref{lem-scount} show that 
$$
|V_i| \ge \mu(\supp(G_i))\sscdot \frac{|F||\Phi|}{|G_i|}. 
$$

Let us choose for each $i$ a subset $W_i\subset V_i$ such that $|W_i| = \mu(\supp(G_i))\sscdot \frac{|F||\Phi|}{|G_i|}$, and let $W$ be the span of those spaces $Y(\phi,p,a)$ such that $(\phi,p,a)$ belongs to $W_i$ for some $i$.

Let $T_W$ be the restriction of $\rho(T)$ to the $\rho(T)$-invariant space $W$.  By Lemmas~\ref{lem-act}  and~\ref{lm-props-of-Y} we have
\begin{multline}\label{eq-fin}
\dim_\k\ker T_W = \sum_{i=1}^N |W_i|\sscdot \dim_\k \ker M(G_i,T;\k) = \\
=\sum_{i=1}^N \mu(\supp(G_i))\sscdot \frac{|F||\Phi|}{|G_i|} \dim_\k \ker M(G_i,T;\k).
\end{multline}

Note that if $(\phi,p,a)\ssin V_i$ then in particular $\dim_\k Y(\phi,p,a) = |G_i|$. It follows that the dimension of the space $W$ is at least
\beq\label{eq-g}
\sum_{i=1}^N |W_i|\sscdot |G_i| =\sum_{i=1}^N \mu(\supp(G_i))\sscdot |\Phi||F| > (1{-}\eps)\sscdot |\Phi||F|.
\eeq

Hence $W$ is a subspace of  $\k[\wb\Phi\sscdot F]$ of codimension at most $2\eps\sscdot|\Phi||F|$. It follows that the codimension of $W\cap \k[\Phi\sscdot F]$ in $\k[\Phi\sscdot F]$ is at most $2\eps\sscdot|\Phi||F|$ as well. On the other hand the codimension of $W\cap \k[\Phi\sscdot F]$ in $W\subset \k[\wb\Phi\sscdot F]$ is at most $\eps\sscdot |\Phi||F|$. It follows that 
$$
|\dim_\k \ker T_W - \dim_\k\ker T_{\Phi\sscdot F}| \le 3\eps,
$$
which together with~\eqref{eq-fin} establishes~\eqref{eq-todo}, which finishes the proof.
\end{proof}

\section{Applications}\label{sec-app}

Now that the computational tool, Theorem \ref{thm:compu-tational-tool}, is available, we can prove Theorems \ref{thm:intro-finitely-generated}, \ref{thm:intro-lamplighter3} and \ref{thm-intro-positive} along the lines described in Remark~\ref{rem-pu}(i). For the sake of  streamlining the discussion in this section, all the linear algebra computations are gathered in Section \ref{sec-lin-algebra}.

In order to prove Theorem \ref{thm:intro-finitely-generated}, we chose to use the group ring elements from  \cite{grabowski-on-turing-dynamical-systems-and-the-atiyah-problem}, but the approach from \cite{arxiv:pichot_schick_zuk-2010} could be used just as well.

The group ring element used in the proof of Theorem \ref{thm:intro-lamplighter3} is also taken from \cite{grabowski-on-turing-dynamical-systems-and-the-atiyah-problem}. It is unclear whether \cite{arxiv:pichot_schick_zuk-2010} could be used to construct a \textit{universal} cover with an irrational $\k$-homology gradient, for the reason explained in Remark~\ref{rem-why}. The group ring element  from \cite{grabowski-on-turing-dynamical-systems-and-the-atiyah-problem} can be used because the group $\ZZ_2\wr \ZZ$ can be explicitly embedded into a finitely presented $2$-step solvable group (see \cite{ Grigorchuk_Linnel_Schick_Zuk}).

The group ring element used in the proof of Theorem \ref{thm-intro-positive} is taken from \cite{arXiv:grabowski-2010-2}.

\subsection{Proof of Theorem \ref{thm:intro-finitely-generated}}\label{subsec-nice}\mbox{}

Let $a\ssin \N$ and  $b\in \{0,1\}$. A \textit{loop-tree graph of type $(a,b)$} is an oriented rooted graph which arises from a tree by adding self-loops at some vertices, such that 
\begin{enumerate}
\item all edges which are not self-loops are directed towards the root,
\item if a vertex is \textit{internal}, i.e.~it has both incoming and outgoing edges in the underlying tree, then it has a self-loop,
\item the root has a self-loop if and only if $b=1$, and 
\item if there are $l$ \textit{leaves}, i.e.~vertices with no incoming edges in the underlying tree, then exactly $l{-}a$ of them have self-loops.
\end{enumerate}

The loops at leaves will be called \textit{external loops}, the loops at internal vertices will be called \textit{internal loops}, and the loop at the root will be called the \textit{root loop}.

If $G$ is a loop-tree graph of type $(a,b)$ and $\k$ is a field then we define two graphs $\al(G;\k)$ and $\be(G;\k)$ with edges labeled by $\k$, as follows. The sets of vertices of both $\al(G;\k)$ and $\be(G;\k)$ are equal to the set $V(G)$ of vertices of $G$. We let $(v,w,s)$ be an edge of  $\al(G;\k)$  if $(v,w)$ is an edge of $G$ and $s=1$. 

The edges of $\be(G;\k)$ are defined depending on the value of $b$, as follows. If $b=0$ then we let $(v,w,s)$ be an edge of $\be(G;\k)$ if $(v,w)$ is an edge of $G$ and $s=0$. If $b=1$ then we let $(v,w,s)$ be an edge of $\be(G;\k)$ if $(v,w)$ is an edge of $G$ and either (i) $(v,w)$ is not the root self loop and $a=0$, or (ii) $(v,w) = (v,v)$ is the root loop and $a=1$. 

Let $M(\al;\k)$  and $M(\be;\k)$ be the adjacency operators on, respectively, $\al(G;\k)$ and $\be(G;\k)$. Thus both  $M(\al;\k)$ and $M(\be;\k)$ are $\k$-linear operators on $\k[V(G)]$.

\begin{figure}[h]%
  \resizebox{0.80\textwidth}{!}{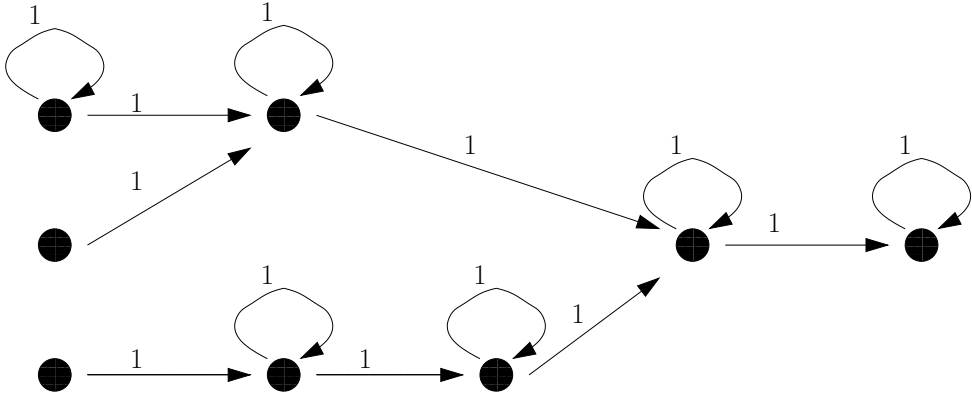}
  \caption{The graph $G_\al$ corresponding to a loop-tree graph of type $(2,1)$}
  \label{fig-g-alpha}
\end{figure}

The proof of the following lemma is deferred to Section \ref{sec-lin-algebra}.
\begin{lemma}\label{lem:nice-graph-indep-of-k}
Let $G$ be a loop-tree graph of type $(a,b)$. Then
$$
    \dim_\k (\ker M(\al;\k) \cap \ker M(\be;\k))  = a - b.
$$
In particular the left-hand side is independent of $\k$. \qed
\end{lemma}

\begin{theorem}\label{thm-black-box-sigma}
Let $A := \oplus_\Z \ZZ_2$. There exists an amenable group $\Ga$ such that for every $\Si\,{\subset}\, \N$ there exists an action $\Ga\ssactson A$, elements $S,T\in \Z_2 [\Ga\ltimes A]$ and an exhausting sequence $G_1,G_2,\ldots$ of simply-connected  $(S,T)$-graphs such that 
$$
 \sum_i \frac{\mu(\supp(G_i))}{|G_i|}\dim_\Q \Big( \ker M(G_i,S;\Q)\cap \ker M(G_i,T;\Q)\Big) =  \frac1{64}- \frac2{8^3} \sum_{i\in \Si} \frac1{2^i}.
$$
Furthermore, for each $i$ there is a loop-tree graph $G$ of some type such that for all fields $\k$ we have $\cal S(G_i,S;\k) = \al(G;\k)$ and $\cal S(G_i,T;\k) = \be(G;\k)$.
\end{theorem}
\begin{proof}
This result can be deduced from the proof of \cite[Theorem 4.3]{grabowski-on-turing-dynamical-systems-and-the-atiyah-problem} and the example in \cite[Section 5.1]{grabowski-on-turing-dynamical-systems-and-the-atiyah-problem}. 

The only difficulty in the derivation of the above theorem is that the group $\Ga$ from \cite[Section 5.1]{grabowski-on-turing-dynamical-systems-and-the-atiyah-problem} is defined as
$$
\Ga := [(\aut(M_1)\wr\ZZ) \ast \ZZ_2]\times \aut(M_2),
$$
where $\ast$ denotes the free product, and $M_1$ and $M_2$ are some finite groups. As such $\Ga$ is not amenable. However, in fact only the subgroup
$$
\Ga_1 := [(\langle\be'\rangle\wr\ZZ) \ast \ZZ_2]\times \aut(M_2),
$$
is used in \cite[Section 5.1]{grabowski-on-turing-dynamical-systems-and-the-atiyah-problem}, where $\be'$ is certain element of $\aut(M_1)$ defined in \cite[Section 5.1]{grabowski-on-turing-dynamical-systems-and-the-atiyah-problem}. Furthermore, a direct check shows that the action $\Ga_1\actson A$ factors through the action of the \textit{amenable} group 
$$
\Ga_2:= [\langle\be'\rangle\times \ZZ_2)\wr\ZZ]\times \aut(M_2),
$$
and so in all of \cite[Section 5.1]{grabowski-on-turing-dynamical-systems-and-the-atiyah-problem} the amenable group $\Ga_2$ can be used instead of $\Ga$.
\end{proof}

To stress that we take an action $\Ga\ssactson A$ guaranteed by the previous theorem for a set $\Si\subset \N$, we will denote the corresponding semidirect product by $\Ga\ltimes_\Si A$.

\begin{cory}\label{cory-main-proposition} Let $\k$ be a field such that $\chr(\k)\neq 2$, let $\Si\subset \N$ and let $S$ and $T$ be as in the previous theorem. Then
\beq\label{eq-bside}
\dim_{\k[\Ga\ltimes_\Si A]} \ker \left(
\left[\begin{array}{c}
\rho(S)\\
\rho(T)
\end{array}\right]
\right)
=
\frac1{64}- \frac2{8^3} \sum_{i\in \Si} \frac1{2^i}.
\eeq
\end{cory}

\begin{proof}
Theorem \ref{thm:compu-tational-tool} shows that the left hand side of \eqref{eq-bside} is equal to
$$
\sum_i \frac{\mu(\supp(G_i))}{|G_i|}\dim_\k ( \ker M(G_i,S;\k)\cap \ker M(G_i,T;\k) ). 
$$
This, by Lemma \ref{lem:nice-graph-indep-of-k}, is equal to 
$$
\sum_i \frac{\mu(\supp(G_i))}{|G_i|}\dim_\Q ( \ker M(G_i,S;\Q)\cap \ker M(G_i,T;\Q) ), 
$$
which by Theorem~\ref{thm-black-box-sigma} is equal to the right hand side of~\eqref{eq-bside}. This finishes the proof.
\end{proof}

The following theorem and Proposition~\ref{prop-eck} establish Theorem \ref{thm:intro-finitely-generated}.

\begin{theorem}
Let $\k$ be a field such that $\chr(\k)\neq 2$. For every non-negative real number $r$ there exists an amenable group $\La$ and a matrix  $T$ over $\Z_2[\La]$ such that 
$$
\dim_{\k[\La]} \ker T  = r.
$$ 
\end{theorem}

\begin{proof}
The previous corollary establishes the theorem for $r$ in the interval  $[\frac1{64}- \frac2{8^3},\frac1{64}]$. By Lemma \ref{lm-appendix-addition}, we see that there exists a natural number $N$ such that the theorem is true for all $r\ge N$.

Thus for every positive real number there exists a natural number $k$ such that the theorem is true for  $2^k \cdot r$. Now Lemma \ref{lm-appendix-multiplication} shows that the theorem is true for $r$ as well. This finishes the proof.
\end{proof}

\subsection {Proof of Theorem \ref{thm:intro-lamplighter3}}\mbox{}

Let  $A := (\oplus_\Z \ZZ_2)^3\times \ZZ_2^3$, let $\Ga := \ZZ^3 \times \aut(\ZZ_2^3)$, and let $\Ga\ssactson A$ be the obvious action: $\aut(\ZZ_2^3)$ acts on $\ZZ_2^3$ and each copy of $\ZZ$ acts by shifting the coordinates of the corresponding copy of $\oplus_\Z\ZZ_2$. The following can be deduced from the proof of \cite[Theorem 4.3]{grabowski-on-turing-dynamical-systems-and-the-atiyah-problem} and the example in \cite[Section 5.3]{grabowski-on-turing-dynamical-systems-and-the-atiyah-problem}.

\begin{theorem}\label{thm-black-box-lamplighter3}
 There exist $S,T\in \Z_2 [\Ga\ltimes A]$ and a sequence $G_1,G_2,\ldots$ of exhausting simply-connected $(S,T)$-graphs such that  
$$
\sum_i \frac{\mu(\supp(G_i))}{|G_i|}\dim_\Q ( \ker M(G_i,S;\Q)\cap \ker M(G_i,T;\Q) ) =     \frac1{64}- \frac18\sum_{k=1}^\infty \frac1{2^{k^2+4k+6}}.
$$
Furthermore, for each $i$ there is a loop-tree graph $G$ of some type such that for all fields $\k$ we have $\cal S(G_i,S;\k) = \al(G;\k)$ and $\cal S(G_i,T;\k) = \be(G;\k)$.\qed
\end{theorem}

The following theorem and Proposition~\ref{prop-eck} establish Theorem \ref{thm:intro-lamplighter3}.

\begin{theorem}
Let $\k$ be a field such that $\chr(\k) \neq 2$. There is matrix $U$ over $\Z_2[(\ZZ_2\wr \ZZ)^3]$ such that 
\beq\label{eq-number-in-question}
\dim_{\k[(\ZZ_2\wr \ZZ)^3]}\ker U =  \frac1{64}- \frac18\sum_{k=1}^\infty \frac1{2^{k^2+4k+6}}.
\eeq
\end{theorem}

\begin{proof}
Note that $\Ga \ltimes A$ is isomorphic to the product of $(\ZZ_2\wr \ZZ)^3$ and the finite
group $\ZZ_2^3\wr \aut(\ZZ_2^3)$. Therefore, by  Lemmas \ref{lm-appendix-addition} and \ref{lm-appendix-matrices}, it suffices to show that the right hand side of  \eqref{eq-number-in-question} is equal to the kernel gradient of a matrix over $\k[\Ga \ltimes A]$.

Let $S$ and $T$ be from the previous theorem. We claim that 
\beq\label{eq-aside}
    \dim_{\k[\Ga \ltimes A]} \ker \left( \begin{array}{c} \rho(S) \\ \rho(T) \end{array}\right) =  \frac1{64}- \frac18\sum_{k=1}^\infty \frac1{2^{k^2+4k+6}}.
\eeq

Theorem \ref{thm:compu-tational-tool} show that the left-hand side of~\eqref{eq-aside} is equal to
$$
\sum_i \frac{\mu(\supp(G_i))}{|G_i|}\dim_\k ( \ker M(G_i,S;\k) \cap \ker M(G_i,T;\k) ). 
$$
This, by Lemma \ref{lem:nice-graph-indep-of-k}, is equal to 
$$
\sum_i \frac{\mu(\supp(G_i))}{|G_i|}\dim_\Q ( \ker M(G_i,S;\Q) \cap \ker M(G_i,T;\Q) ). 
$$
which by the previous theorem is equal to the right-hand side of~\eqref{eq-aside}. This finishes the proof.
\end{proof}

\subsection{Proof of Theorem \ref{thm-intro-positive}}\label{subsec-coaster}\mbox{}

\newcommand{\ov}{\overline}

Let $\k$ be a field. For $k,l\ssin \N_+$ we introduce three graphs, $G_1(k;\k)$,  $G_2(l;\k)$ and $G_3(k,l;\k)$, with edges labeled by $\k$.

Let $G_1(k;\k)$ be the graph with $2k$ vertices which arises from the graph in Figure \ref{fig_sgraph_g} by adding a self-loop with label $1$ at each vertex.

Let $G_2(l;\k)$ be the graph with $2l+1$ vertices which arises from the graph in Figure \ref{fig_sgraph_h} by adding a self-loop with label $1$ at each vertex except for the unique vertex with no outgoing edges. 

Finally, let $G_3(k,l;\k)$ be the graph with $2k+2l+2$ vertices which arises from the graph  in Figure \ref{fig_sgraph_j} by adding a self-loop with label $1$ at each vertex, except for the two
unique vertices with either no outgoing edges or no incoming edges.

\begin{figure}[h]%
  \resizebox{0.79\textwidth}{!}{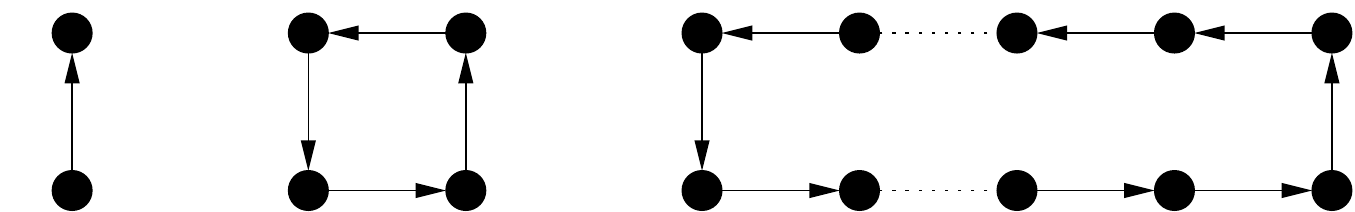}
  \caption{}
  \label{fig_sgraph_g}
\end{figure}

\begin{figure}[h]%
  \resizebox{0.70\textwidth}{!}{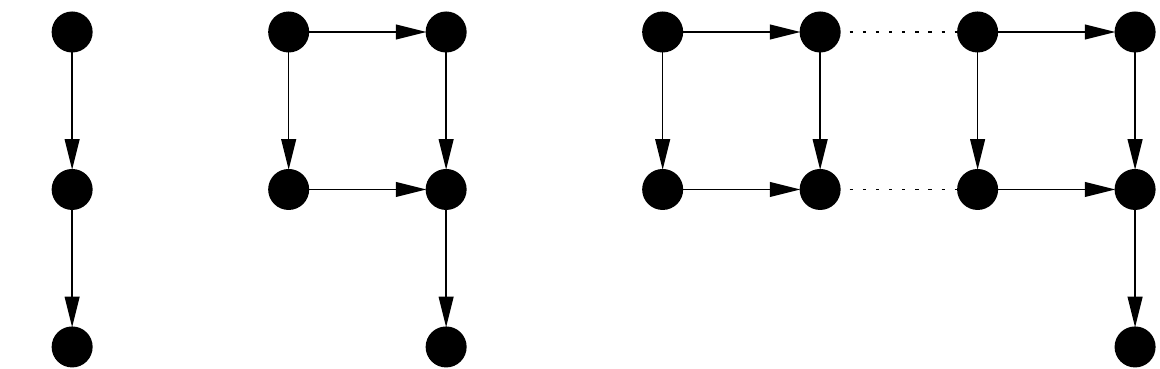}
  \caption{}
  \label{fig_sgraph_h}
\end{figure}

\begin{figure}[h]%
  \resizebox{0.89\textwidth}{!}{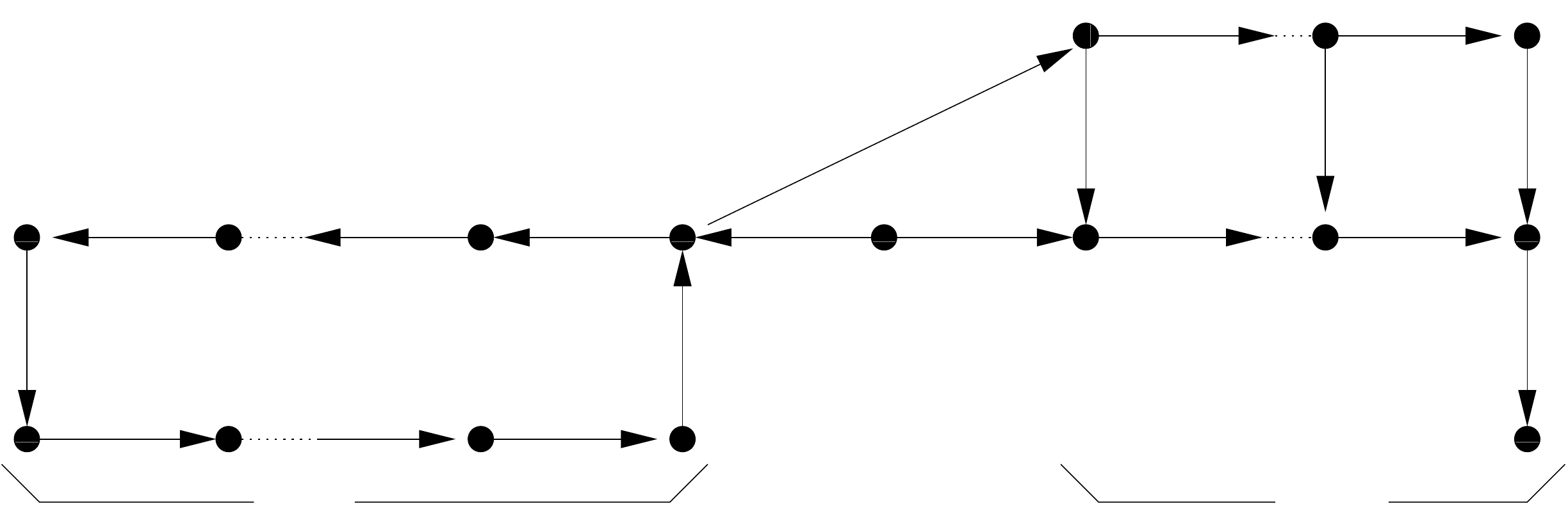}
  \caption{}
  \label{fig_sgraph_j}
\end{figure}

Let $M_1(k;\k)$, $M_2(l;\k)$ and $M_3(k,l;\k)$ be the adjacency operators on, respectively, the graphs $G_1(k;\k)$,  $G_2(l;\k)$ and $G_3(k,l;\k)$. The proof of the following lemma is deferred to Section \ref{sec-lin-algebra}.

\begin{lemma}\label{lem-dims}We have
\begin{alignat*}{2}
 &\dim_\k \ker M_1(k;\k) &&= \,\, \left\{ 
	\begin{array}{l l}
  		1 &   \mbox{if $k>1$ and $1 = 2^{k-1}$ in $\k$,}\\
		0 &   \mbox{otherwise,}\\ 
	\end{array} \right. 
\\[5pt]
&\dim_\k \ker M_2(l;\k)  &&= \,\, 1,
\\[5pt]
 &\dim_\k \ker M_3(k,l;\k) &&=  \,\,\left\{ 
	\begin{array}{l l}
  		2 &   \mbox{if $l= 2^{k-1}-1$ modulo $\chr(\k)$,}\\
		1 &   \mbox{otherwise.}\\ 
	\end{array} \right. 
\end{alignat*}\qed
\end{lemma}

Let  $A := (\oplus_\Z \ZZ_2)\times \ZZ_2^3$, let $\Ga := \ZZ \times \aut(\ZZ_2^3)$, and let $\Ga\ssactson A$ be the obvious action. The following can be deduced from \cite[Section 5]{arXiv:grabowski-2010-2}.

\begin{theorem}\label{thm-black-box-coaster}
 There is $T\in \Z_2 [\Ga\ltimes A]$ and $T$-graphs $U$, $G_1(k)$ for $k\ssin \N_+$, $G_2(l)$ for $l\ssin \N_+$, and $G_3(k,l)$ for $k,l\ssin \N_+$, with the following properties.

\begin{enumerate}
\item  There exists an exhausting sequence of simply-connected $T$-graphs which enumerates the set $\{U\}\cup \{G_1(k)\colon k\ssin \N_+\}\cup \{G_2(l)\colon l\ssin \N_+\} \cup \{G_3(k,l)\colon k,l\ssin \N_+\}$.
\item For every field $\k$ we have that $\cal S(U,T;\k)$ is a graph with one vertex and no edges, and  there exist  isomorphisms of labeled graphs
 $$
 \cal S(G_1(k),T;\k) \cong  G_1(k;\k), \quad  \cal S(G_2(l),T;\k) \cong  G_2(l;\k),\quad  \cal S(G_3(k,l),T;\k) \cong  G_3(k,l;\k).  
 $$
 
  \item We have 
\begin{align*}
   \quad \mu (\supp(U)) =\frac{45}{64} 
   &\quad\quad\mu(\supp(G_1(k))) = \frac{2k}{64\cdot 2^k}\\
\quad\mu(\supp(G_2(l))) = \frac{2l+1}{64\cdot 2^l}
&\quad\quad\mu(\supp (G_3(k,l))) = \frac{2k+2l+2}{2^{k+l}}.
\end{align*}
\end{enumerate}
\qed
\end{theorem}

Let $q:=\chr(\k)>2$, let $r$ be the multiplicative order of $2$ in $\k^\ast$, let $L\subset \{2,3,\ldots \}$ be the finite set of minimal
representatives in $\{2,3,\ldots\}$ of the powers $2^i$ mod $q$, $i\ge 0$, and for $x\in L$ let
$r(x)\ge 0$ be the smallest natural number such that $2^{r(x)} =x \mod q$.

\begin{theorem}\label{thm-rem}
Let $\k$ be a field such that $\chr(\k)\neq 2$ and let $T$ be as in the previous theorem. We have 
$$
 \dim_{\k[\Ga\ltimes A]}\ker \rho(T) = \frac{47}{64} + \frac{1}{128}\sscdot \frac{1}{2^r-1} + \frac{1}{64}\sscdot \frac{2^q}{2^q-1} \sscdot \frac{2^r}{2^r-1}\sscdot\sum_{x\in L} \frac{1}{2^{x+r(x)}}.
$$
\end{theorem}

\begin{remark}The additional factor $1344$ in the statement of Theorem \ref{thm-intro-positive} in the introduction arises from the fact that $\Ga\ltimes A$ is isomorphic to $\ZZ_2 \wr \ZZ \times (\ZZ_2^3\rtimes \aut(\ZZ_2^3))$, and the order of $\ZZ_2^3\rtimes \aut(\ZZ_2^3)$ is $1344$ (see Lemma~\ref{lm-appendix-matrices}).  

On the other hand, in \cite{ Grigorchuk_Linnel_Schick_Zuk} the group $\ZZ_2\wr \ZZ$ is explicitly embedded into a finitely presented $2$-step solvable group $\La$, and by Lemma~\ref{lm-appendix-subgroups} we have $\dim_{\k[\Ga\ltimes A]} \ker \rho(T) = \dim_{\k[\La]} \ker \rho(T)$. As such Theorem~\ref{thm-rem} and Proposition~\ref{prop-eck} establish Theorem~\ref{thm-intro-positive}.
\end{remark}

\begin{proof}[Proof of Theorem~\ref{thm-rem}]
By Theorem \ref{thm:compu-tational-tool}, $\dim_{\k[\Ga\ltimes A]}\ker \rho(T)$ is equal to the sum of the terms
\begin{align*}
&\sum_{k\ge 1} \frac{\mu(\supp (G_1(k)))}{2k} \cdot \dim_\k\ker M(G_1(k),T;\k),   \\
&\sum_{l\ge 1} \frac{\mu(\supp (G_2(l)))}{2l+1} \cdot \dim_\k\ker M(G_2(l),T;\k), \\
&\sum_{k,l\ge 1} \frac{\mu(\supp (G_3(k,l)))}{2k+2l+2} \cdot\dim_\k\ker M(G_3(k,l),T;\k),
\end{align*}
and the number $\frac{45}{64}$, corresponding to the $T$-graph $U$ in Theorem \ref{thm-black-box-coaster}. By Lemma \ref{lem-dims} the sums above are equal to, respectively,
\begin{gather*}
\sum_{\mathclap{\substack{k>1, \\ 2^{k-1}=1(\k)}}} \, \frac{1}{64}\sscdot\frac{1}{2^k} = 
    \frac{1}{64}\sum_{n=1}^\infty \frac{1}{2^{rn+1}} 
    = \frac{1}{128}\left(\frac{1}{1-2^{-r}} -1\right) = \frac{1}{128}\sscdot \frac{1}{2^r-1}, \\
\sum_{l\ge 1} \, \frac{1}{64} \sscdot \frac{1}{2^l} = \frac{1}{64},
\end{gather*}
and 
\begin{align*}
    \frac{1}{64}\sum_{k,l\ge 1}   \frac{1}{2^{k+l}} + \frac{1}{64} \sum_{\mathclap{\substack{k,l\ge 1,\\l=2^{k-1}-1\,(\k)}}} \frac{1}{2^{l+k}} \,&=\, 
    \frac{1}{64} +\frac{1}{64} \sum_{x\in L} \sum_{\substack{n,m\ge 0}}
    \frac{1}{2^{(x+qn-1)+ (r(x)+r m+1)}} =\\
   & = \frac{1}{64} +
    \frac{1}{64}\sscdot \frac{2^q}{2^q-1}\sscdot \frac{2^r}{2^r-1} \sscdot \sum_{x\in L}
    \frac{1}{2^{x+r(x)}},
\end{align*}
and so the theorem follows.
\end{proof}

\section{Linear algebra computations}\label{sec-lin-algebra}

The notation in the following lemma is as in Subsection~\ref{subsec-nice}.

\begin{lemma}
Let $G$ be a loop-tree graph of type $(a,b)$. Then
$$
    \dim_\k (\ker M(\al;\k) \cap \ker M(\be;\k))  = a - b.
$$
In particular the left-hand side is independent of $\k$.
\end{lemma}

\begin{proof}
Let $r\ssin V(G)$ be the root. For $v\in V(G)$ let $n(v) :=v$ if $v =r$, and otherwise let $n(v)$ be the unique vertex different from $v$ such that there is a directed edge from $v$ to $n(v)$. 

For $v\ssin V(G)$ let us define $\xi(v)\ssin \k[V(G)]$ to be equal to $v-n(v) +n^2(v)-\ldots \pm r$. A direct check shows that 
$$
\ker M(\al;\k)  = \lin_\k (\xi(v)\colon\, v\text{ is a leaf vertex without external loop}),
$$
and so we deduce that $\dim \ker M(\al;\k)  =a$. On the other hand clearly if $b=0$ then $M(\be;\k) $ is the zero operator and so the claim holds. If $b=1$ then $\ker M(\be;\k) $ is a codimension $1$ subspace which is not contained in $\ker M(\al;\k) $, and so in this case the claim also holds.
\end{proof}

The notation in the following lemma is as in Subsection~\ref{subsec-coaster}.

\begin{lemma}
We have
\begin{alignat*}{2}
 &\dim_\k \ker M_1(k,\k) &&= \,\, \left\{ 
	\begin{array}{l l}
  		1 &   \mbox{if $k>1$ and $1 = 2^{k-1}$ in $\k$,}\\
		0 &   \mbox{otherwise,}\\ 
	\end{array} \right. 
\\[5pt]
&\dim_\k \ker M_2(l,\k)  &&= \,\, 1,
\\[5pt]
 &\dim_\k \ker M_3(k,l;\k) &&=  \,\,\left\{ 
	\begin{array}{l l}
  		2 &   \mbox{if $l= 2^{k-1}-1$ modulo $\chr(\k)$,}\\
		1 &   \mbox{otherwise.}\\ 
	\end{array} \right. 
\end{alignat*}
\end{lemma}

\begin{proof}[Sketch of Proof] All the equalities are elementary and we show only the third one, which needs the longest argument. We will assume that $k>1$. The arguments in the case $k=1$ are very similar and are left to the reader.

Let us denote $G:=G_3(k,l;\k)$ and  let us  give the vertices of $G$ names as in Figure \ref{fig_graph_j}. 
\begin{figure}[h]%
  \resizebox{\textwidth}{!}{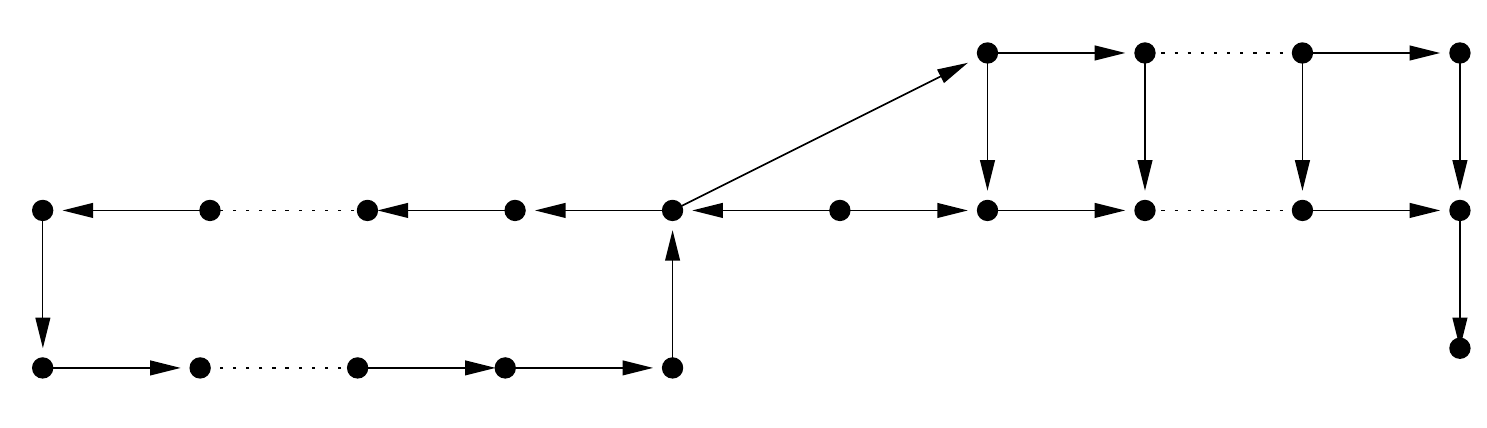}
  \caption{}
  \label{fig_graph_j}
\end{figure}

Let us denote $M :=M_3(k,l;\k)$. First, let us assume  that $l= 2^{k-1}-1$ modulo $\chr(\k)$. The first generator of $\ker (M)$ is the indicator function of the vertex $F$. The coefficients of the second generator of $\ker( M)$ are depicted in Figure \ref{fig_graph_j_ker}.
\begin{figure}[h]%
  \resizebox{\textwidth}{!}{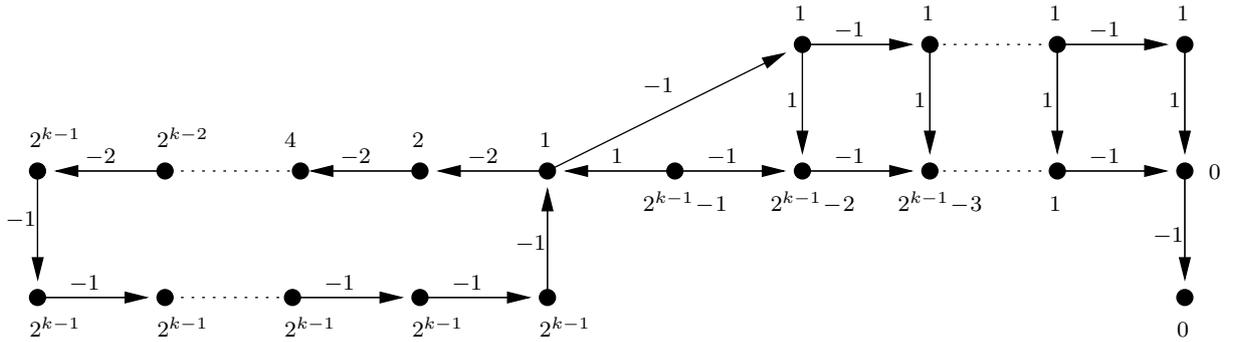}
  \caption{Coefficients of the second generator of $\ker( M)$ when $l= 2^{k-1}-1$ modulo $\chr(\k)$.}
  \label{fig_graph_j_ker}
\end{figure}

To see that these two vectors generate all of $\ker (M)$ let us prove the following.

\begin{lemma}
Let $f\ssin \ker (M)$ be such that $f(F) = f(A_1) = 0$. Then $f=0$.
\end{lemma}
\begin{proof}
Since $f\ssin\ker (M)$, for every vertex $v$ we have
\beq\label{eq-obvious}
	\sum_{w,a\colon(w,v,a)\ssin E(G)}  a \sscdot f(w) = 0,
\eeq
The equation~\eqref{eq-obvious} at $v = A_2$ shows that if $f(A_1)=0$ then also $f(A_2)=0$. Similarly, we see  that $f(A_i) = f(B_i)= 0$ for all $i=1,\ldots, k$. Now the equation~\eqref{eq-obvious}  at $v = A_1$ together with the equalities $f(A_1) = f(B_k)=0$ imply that $f(I)=0$. Similarly, the equation~\eqref{eq-obvious} at $v = C_1$ and the fact that $f(A_1)=0$ imply that $f(C_1)=0$. Now, the equation~\eqref{eq-obvious} at $v = D_1$ together with the equalities $f(I) = f(C_1)=0$ imply that  $f(D_1)=0$.

Now we note that the equation~\eqref{eq-obvious} at $v = C_{i+1}$ and the equality $f(C_i) =0$ imply that $f(C_{i+1})=0$. Thus we get $f(C_i)= 0$ for all $i=1,\ldots, l$. Finally, the equation~\eqref{eq-obvious} at $v = D_{i+1}$ and the equalities $f(D_i) = f(C_{i+1}) =0$ imply that $f(D_{i+1}) =0$, and so we deduce that  $f(D_i) =0$ for all $i=1,\ldots, l$. Since $f(F)=0$ by assumption, the lemma follows.
\end{proof}

Note that the indicator function of the vertex $F$ is in $\ker (M)$ for
arbitrary $(k,l)$. Thus to finish the proof it is enough to show that if $f\ssin
\ker(M)$ is such that $f(A_1)= 1$ and $f(F)=0$ then $l= 2^{k-1}-1$ modulo $\chr(\k)$.

Thus, let $f\ssin \ker(M)$ be such that $f(A_1)= 1$ and $f(F)=0$. The equation~\eqref{eq-obvious} at $v=A_2$ implies that $f(A_2)= 2$. Similarly we obtain $f(A_i) = 2^{i-1}$ for all $i=1,\ldots, k$, and so in particular we have that $f(A_k) =2^{k-1}$.  Now the equation~\eqref{eq-obvious} at  $v = B_1$ shows that $f(B_1)=2^{k-1}$. Similarly we obtain $f(B_i) = 2^{k-1}$ for all $i=1,\ldots, k$, and so in particular $f(B_k)=2^{k-1}$.

Since $f(A_1) =1$ and $f(B_k) =2^{k-1}$ the equation~\eqref{eq-obvious} at $v = A_1$ implies that $f(I) = 2^{k-1}$.  Now the equation~\eqref{eq-obvious} at $v = C_1$ together with the equality $f(A_1) =1$ imply that  $f(C_1)=1$. Similarly we see that $f(C_i)=1$ for all $i=1,\ldots,l$. As such, the equation~\eqref{eq-obvious} at $v = D_1$ implies that $f(D_1) =2^{k-1}-2$ and similarly we see that $f(D_i) = 2^{k-1}-i-1$ modulo $\chr(\k)$, for all $i=1,\ldots, l$. In particular, we see that $f(D_l) =0 $ only if $l= 2^{k-1}-1$ modulo $\chr(\k)$. Since the equation~\eqref{eq-obvious} at $v = F$ implies that $f(D_l)= 0$, this finishes the proof.  
\end{proof}

{\footnotesize \subsection*{Acknowledgements} Ł.G. was supported by EPSRC at Imperial College London and Oxford University, by EPSRC grant EP/K012045/1 at University of Warwick, by Austrian Science Foundation project P25510-N26 during Ł.G.'s stay at T.U.~Graz, and by Fondations Sciences Math{\'e}matiques de Paris during the program \textit{Marches Aléatoires et Géométrie Asymptotique des Groupes} at Institut Henri-Poincaré.

Both authors were supported by the Erwin Schrödinger International Institute for Mathematical Physics during the conference \textit{Measured Group Theory} in February 2016.

Both authors would like to thank Johannes Neumann for pointing out inaccuracies in the proof of Theorem~\ref{thm:compu-tational-tool} in a previous version of this article.
}

\bibliographystyle{alpha}
\bibliography{bibliografia}

\end{document}